\documentclass[a4paper,11pt,reqno]{amsart}

\RequirePackage[l2tabu, orthodox]{nag}
\usepackage[T1]{fontenc}
\usepackage[utf8]{inputenc}
\usepackage{lmodern}
\usepackage[frenchb,english]{babel}
\usepackage{color}
\usepackage{amsthm,amssymb,amsmath}
\usepackage{graphicx}
\usepackage[all]{xy}
\usepackage{enumerate}
\usepackage{pst-node}
\usepackage{tikz-cd} 
\usepackage{color}
\usepackage{bm}

\def\comment#1{}

\newcommand{\R}{\mathbb{R}}
\newcommand{\N}{\mathbb{N}}
\newcommand{\F}{\mathbb{F}}

\newcommand{\U}{\mathbb{U}}

\newcommand{\cG}{\mathcal{G}}

\newcommand{\T}{\mathbb{T}}
\newcommand{\cE}{\mathcal{E}}
\newcommand{\cC}{\mathcal{C}}

\newcommand{\cM}{\mathcal{M}}
\newcommand{\cL}{\mathcal{L}}

\newcommand{\cT}{\mathcal{T}}

\newcommand{\x}{\mathbf{x}}

\newcommand{\fhi}{\varphi}

\def\to{\mathop{\rightarrow}}
\def\dans{\mathop{\subset}}
\newcommand{\moins}{\setminus}

\newcommand{\Homeo}{\mathrm{Homeo}}

\newcommand{\dist}{\mathrm{dist}}
\newcommand{\Int}{\mathrm{Int}}
\newcommand{\Ker}{\mathrm{Ker}}
\newcommand{\Id}{\mathrm{Id}}
\newcommand{\cN}{\mathcal{N}}

\newcommand{\vide}{\emptyset}
\def\dans{\mathop{\subseteq}}

\newtheorem{theorem}{Theorem}[section]
\newtheorem{maintheorem}{Theorem}

\newtheorem{lemma}[theorem]{Lemma}
\newtheorem{proposition}[theorem]{Proposition}

\newtheorem{corollary}[theorem]{Corollary}

\theoremstyle{definition}
\newtheorem{definition}[theorem]{Definition}

\theoremstyle{remark}
\newtheorem{remark}[theorem]{Remark}

\usepackage[margin=3.5cm]{geometry}

\usepackage[colorlinks=true,linkcolor=blue,citecolor=magenta]{hyperref}

\title[Topology of leaves for minimal laminations by hyperbolic surfaces II]{Topology of leaves for minimal laminations by non-simply connected hyperbolic surfaces}

\date{}
\author[S. Alvarez]{S{\'e}bastien Alvarez}\thanks{Corresponding author: Sébastien Alvarez}
\address{CMAT, Facultad de Ciencias, Universidad de la Rep\'ublica, Uruguay}
\email{salvarez@cmat.edu.uy}

\author[J. Brum]{Joaqu\'in Brum}
\address{IMERL, Facultad de Ingenier\'ia, Universidad de la Rep\'ublica, Uruguay}
\email{jbrum@fing.edu.uy}

\thanks{The authors were partially supported by CSIC 618, CSIC I+D 389, FCE-135352, FCE-148740 and MathAmSud RGSD 19-MATH-04 as well as by Distinguished Professor Fellowships of FSMP. S.A. acknowledges the support of LIA-IFUM. J.B. acknowledges the support of CONICYT via FONDECYT Postdoctorate 3190719.}

\begin{document}

\maketitle

\begin{abstract}
We give the topological obstructions to be a leaf in a minimal lamination by hyperbolic surfaces whose generic leaf is homeomorphic to a Cantor tree. Then, we show that all allowed topological types can be simultaneously embedded in the same lamination. This result, together with results in \cite{ABMP} and \cite{Blanc_2bouts}, complete the panorama of understanding which topological surfaces can be leaves in minimal hyperbolic surface laminations when the topology of the generic leaf is given. In all cases, all possible topologies can be realized simultaneously. 
\smallskip
{\noindent\footnotesize \textbf{MSC\textup{2010}:} Primary 57R30. Secondary  37B05.}

\end{abstract}

\section{Introduction}

A \emph{surface lamination} is a compact and metrizable topological space $\cL$ locally modeled on the product of the unit disc by a compact set. It comes with an atlas, giving coordinates to these open sets, whose transition functions preserve the disc factor of the product structure. These discs glue together to form surfaces whose global behavior may be very complicated, we call these surfaces the \emph{leaves} of the lamination. We are interested in \emph{minimal} laminations, i.e. those laminations in which every leaf is dense. Note that every lamination contains a minimal lamination. We refer the reader to \cite{candel-conlon} for the general theory of laminations.

The compact factors of the local product srtructure are called \emph{transversals}, when these transversals are homeomorphic to Cantor sets, we say that the lamination is a \emph{solenoid} (see \cite{Sullivan,Verjovsky}). When transition functions are holomorphic along the disc coordinate, we say that $\cL$ is a \emph{Riemann surface lamination} (see \cite{Ghys_Laminations} for the general theory). Finally, Riemann surface laminations where all leaves are of hyperbolic type are called \emph{hyperbolic surface laminations}. In this case, there exists a complete hyperbolic metric in every leaf which varies continuously in the transverse direction (see \cite{Candel}). Hyperbolic surface laminations appear quite frequently and a topological characterization is given by Candel in \cite{Candel}.

Thanks to Cantwell-Conlon \cite{Cantwell_Conlon}, we understand the topology of generic leaves of minimal surface laminations. Here generic means the Baire point of view. According to this work the generic leaf has $1$, $2$ or a Cantor set of ends, and either it has genus zero or every end is accumulated by genus. This gives six possible topological types for a generic leaf. Moreover all the leaves in a dense and saturated residual set are homeomorphic: \cite[Theorem B]{Cantwell_Conlon}. See \cite{Ghys_generic} for a probabilistic counterpart of this theorem. 

The present work is devoted to the topological study of leaves of minimal hyperbolic surface laminations. More precisely, we are interested in describing the possible topologies of non-generic leaves that can occur when the topological type of the generic leaf is given. 

In a companion paper \cite{ABMP}, written together with Mart\'inez and Potrie, we treated the case of minimal hyperbolic solenoids with a simply connected generic leaf. More precisely, we constructed a minimal lamination by hyperbolic surfaces such that every non-compact surface is homeomorphic to a leaf of the lamination. To achieve this we considered the inverse limit of a (very carefully chosen) tower of finite coverings over a closed hyperbolic surface. In his unpublished PhD thesis \cite{Blanc_these} Blanc constructed a similar example, with a completely different method  (inspired by Ghys-Kenyon's construction \cite{Ghys_Laminations}). It is worth mentioning that his lamination is not of hyperbolic type. We also refer to the recent and interesting work of Meni\~no-Gusm\~ao about realization of topological types in leaves of minimal hyperbolic foliations of codimension $1$: see \cite{MG}.

In this paper we treat the case of minimal hyperbolic laminations whose generic leaves are \emph{Cantor trees}, i.e. are homeomorphic to a sphere minus a Cantor set. An example of such object is the classical \emph{Hirsch's foliation} (see \cite{Hirsch1975} for the original construction and for example \cite{AlvarezLessa,Cantwell_Conlon,Ghys_generic} for the minimal model): its leaves are Cantor trees, except countably many, which are homeomorphic to the torus minus a Cantor set. We show that, unlike in the case of simply connected generic leaves, there are topological obstructions to be a leaf of such laminations. Precisely, in Proposition \ref{p.obstpant} we show that \emph{if a minimal hyperbolic lamination has a non simply connected generic leaf then all of its leaves satisfy condition $(\ast)$ defined below.}

\noindent \textbf{Condition $(\ast)$ --} A non-compact surface satisfies \emph{condition $(\ast)$} if its isolated ends are accumulated by genus.\\

Then, we prove that condition $(\ast)$ is the only topological obstruction for being the leaf of such a lamination. Our main result is the following theorem.

\begin{maintheorem}\label{maintheorouille}
There exists a minimal hyperbolic surface lamination $\cL$ such that
\begin{itemize}
\item the generic leaf of $\cL$ is a Cantor tree;
\item every non-compact surface satisfying condition $(\ast)$ is homeomorphic to a leaf of $\cL$.
\end{itemize}
\end{maintheorem}

The method and formalism in the proof of Theorem \ref{maintheorouille} resemble the ones used in \cite{ABMP}. Namely, we construct laminations taking inverse limits of towers of finite coverings. However, in this case we reduce the proof of Theorem \ref{maintheorouille} to that of Theorem \ref{t.grafos} which involves towers of finite coverings of graphs. The idea of using towers of coverings of graphs to get interesting solenoidal manifolds is not new and can be found for example in \cite{McCord}, \cite{Schori} or \cite{CFL}.

However, there is a big difference between the proofs of Theorem \ref{t.grafos} and those appearing in \cite{ABMP}: in this case, due to the very combinatorial nature of the problem, we cannot prefix the topological types of the leaves that we want to embed in the graph lamination. For this reason we need to define new objects called \emph{$C$-graphs} which represent graphs up to some local information that does not affect the (large scale) topological invariants that we need to realize in our leaves. It turns out that we can construct a ``big'' family of such $C$-graphs, realizing the desired topological invariants and which we can simultaneously ``realize'' inside an inverse limit lamination. See Section \ref{ss.strategy} for a more precise outline of our strategy.

\paragraph{\textbf{Other generic leaves}} A \emph{Cantor tree with handles} is by definition a non-compact and orientable surface having a Cantor set of ends, each of which being accumulated by genus. By performing a surgery, we obtain in Section \ref{s.corollary} the following corollary.

\begin{corollary}\label{corollary}
There exists a minimal hyperbolic surface lamination $\cL'$ such that
\begin{itemize}
\item the generic leaf of $\cL'$ is a Cantor tree with handles;
\item for every non-compact surface $\Sigma$ such that every end is accumulated by genus, there exists a leaf of $\cL'$ homeomorphic to $\Sigma$.
\end{itemize}

\end{corollary}

\begin{remark}\label{r.cirugia} Similarly, applying the same construction to the lamination constructed in Theorem A of \cite{ABMP}, we see that the previous result holds if we impose that the generic leaf is an \emph{infinite Loch Ness monster}, i.e. has one end and infinite genus.
\end{remark}

Notice that Proposition \ref{p.obstpant} together with Theorem \ref{maintheorouille} and Corollary \ref{corollary} show the precise obstructions to be a leaf of a minimal lamination by hyperbolic surfaces whose generic leaf has a Cantor set of ends. Moreover, we show how to embed all leaves with allowed topological types simultaneously. On the other hand, Theorem A in \cite{ABMP} together with Remark \ref{r.cirugia} give analogous results for the case where the generic leaf has one end.

Finally, in \cite{Blanc_2bouts} Blanc gives a complete description of which non-compact surfaces can be realized as leaves of minimal laminations by surfaces where the generic leaf has two ends: all leaves have one or two ends. If such a lamination carries a hyperbolic structure then Proposition \ref{p.obstpant} implies that the generic leaf must be a \emph{Jacob ladder} (with two ends, each of which being accumulated by genus) and the only surface that can appear, other than the Jacob ladder, is the Loch-Ness monster.  Blanc builds in \cite[Section 2]{Blanc_2bouts} an example of minimal foliation by surfaces whose leaves are homeomorphic to a Jacob ladder with the exception of four leaves which are homeomorphic to a Loch-Ness monster. Notice that the previous example admits hyperbolic structures because all leaves are of infinite topological type. 

This completes the picture: \emph{we completely understand the possible topologies of leaves of minimal laminations by hyperbolic surfaces in terms of the topology of the generic leaf}. Moreover, for each topological type of the generic leaf, \emph{all possible leaves can appear simultaneously}. This is summarized in Table \ref{table:summary}.

\begin{table}[ht]
	\begin{center}
		\begin{tabular}{|c|c|}
			\hline
			\textbf{Generic leaf} & \textbf{Possible leaves}\\
			\hline
			Disc  & All surfaces\\
			\hline
			Cantor tree  & Surfaces with condition $(\ast)$\\
            \hline
            Loch-Ness monster & Surfaces with ends accumulated by genus\\
            \hline
            Cantor tree with handles & Surfaces with ends accumulated by genus\\
            \hline
            Jacob ladder  & Jacob ladder and Loch-Ness monster\\
            \hline
		\end{tabular}
	\end{center}
	\caption{Possible topologies of leaves of minimal hyperbolic surface laminations.}\label{table:summary}
\end{table}

\subsection{Organization of the paper}In Section \ref{s.prelim} we give basic definitions and notions that will be used throughout the text. Then, in Section \ref{s.graphs} we show how to deduce Theorem \ref{maintheorouille} from an analogous theorem for laminations by graphs (Theorem \ref{t.grafos}). In \S \ref{ss.strategy} we give an informal strategy of the proof. In Section \ref{s.decorouille} we define $C$-graphs and prove their basic properties. Then, in Section \ref{ss.forests} we define forests of $C$-graphs, their limit graphs and their realizations inside towers of finite coverings. In Section \ref{s.surgeries} we define the surgery operation and use it two prove our main Lemma (Lemma \ref{l.main}) saying that some families of forests can be included in towers. In Section \ref{s.proof} we prove Theorem \ref{t.grafos} by including a particular forest of $C$-graphs in a tower but assuming the existence of this object. In Section \ref{s.forestuniv} we give the proof of the existence of the aforementioned forest of $C$-graphs (Proposition \ref{p.forestuniv}). Finally, in Section \ref{s.Appendix} we prove Corollary \ref{corollary} and a simple but highly technical Lemma used in the proof of Proposition \ref{p.forestuniv}.

\section{Preliminaries}\label{s.prelim}
In this section we define basic notions that will be used throughout the text. Also, we show that condition $(\ast)$ is an obstruction to be a leaf of a minimal hyperbolic surface lamination with non-simply connected generic leaf. 
\subsection{Non-compact surfaces and condition $(\ast)$}\label{ss.ncsurfaces}

\paragraph{\textbf{Ends of a space }} Let $G$ be a connected, locally connected and locally compact topological space and $(K_n)_{n\in\N}$ an exhausting and increasing sequence of compact subsets of $G$. An \emph{end} of $G$ is a strictly decreasing and infinite sequence $(\cC_n)_{n\in\N}$ where $\cC_n$ is a connected component of $G\moins K_n$. We denote by $\cE(G)$ the \emph{space of ends} of $G$. It is independent of the choice of $K_n$.

The space of ends of $G$ possesses a natural topology which makes it a totally disconnected, compact and metrizable space. Therefore it is homeomorphic to a closed subset of a Cantor set. To be more precise, take an end $e$ defined by a sequence $(\cC_n)_{n\in\N}$. Then, any open set $V\subseteq G$ such that $\cC_n\subseteq V$ for all but finitely many $n\in\N$ defines a neighbourhood of $e$ consisting of those ends whose defining sequences also lie inside $V$ for all but finitely many $n\in\N$.

\paragraph{\textbf{Classifying triples }} In what follows, a \emph{classifying triple} is the data $\tau=(g,\cE_0,\cE)$ of
\begin{itemize}
\item a number $g\in\N\cup\{\infty\}$;
\item a pair of nested spaces $\cE_0\dans\cE$ where is $\cE_0$ closed and $\cE$ is a nonempty, totally disconnected and compact topological space; satisfying
\item $g=\infty$ if and only if $\cE_0\neq\vide$.
\end{itemize}

We say that two classifying triples $\tau=(g,\cE_0,\cE)$ and $\tau'=(g',\cE_0',\cE')$ are \emph{equivalent} if $g=g'$ and there exists an homeomorphism $h:\cE\to\cE'$ such that $h(\cE_0)=\cE_0'$.

\paragraph{\textbf{Noncompact surfaces }} We now recall the modern classification of surfaces as it appears in \cite{Ric}. We are only interested in orientable surfaces.

Say that an end $e=(\cC_n)_{n\in\N}$ of a surface $\Sigma$ is \emph{accumulated by genus} if for every $n\in\N$, the surface $\cC_n$ has genus. The ends accumulated by genus form a compact subset that we denote by $\cE_0(\Sigma)\dans\cE(\Sigma)$. In our terminology the triple $\tau(\Sigma)=(g(\Sigma),\cE_0(\Sigma),\cE(\Sigma))$ is a classifying triple.

\begin{theorem}[Classification of surfaces]\label{classification}
Two orientable noncompact surfaces $\Sigma$ and $\Sigma'$ are homeomorphic if and only if their classifying triples $\tau(\Sigma)$ and $\tau(\Sigma')$ are equivalent. Moreover for every classifying triple $\tau$ there exists an orientable noncompact surface $\Sigma$ such that $\tau(\Sigma)$ is equivalent to $\tau$.
\end{theorem}

\paragraph{\textbf{Condition $(\ast)$}} Say that a classifying triple $\tau=(g,\cE_0,\cE)$ satisfies \emph{condition $(\ast)$} when every isolated point of $\cE$ belongs to $\cE_0$. 
We also say that the pair $(\cE_0,\cE)$ satisfies condition $(\ast)$.

Say that a surface $\Sigma$ satisfies condition $(\ast)$ when its classifying triple does so. This means that its isolated ends are accumulated by genus. 

\subsection{Hyperbolic surface laminations and towers of coverings}

\paragraph{\textbf{Reeb's stability theorem }} We need the classical Reeb's stability theorem. It is usually stated for foliations (see \cite{Camacho_LinsNeto}), however the proof can be adapted in the laminated context (see also the proof given by Lessa in his thesis: \cite{Lessa}). We refer to any of these two references (\cite{Camacho_LinsNeto} or \cite{Lessa}) for the definition of the holonomy group of a leaf of a lamination.

\begin{theorem}[Reeb's stability theorem]\label{t.Reeb}
Let $\cL$ be a lamination and $U$ be an open subset of some leaf $L$ of $\cL$ with compact closure. Assume that the holonomy germ along every closed path in $U$ is trivial. Then there exists a neighbourhood $W$ of $U$ in $\cL$ and a homeomorphism $\phi:W\to U\times T$, the set $T$ being a transversal to $\cL$, such that $\phi_\ast\cL$ is the trivial lamination $(U\times\{t\})_{t\in T}$.
\end{theorem}

\paragraph{\textbf{Hyperbolic surface laminations }} Let $\cL$ be a compact metric space endowed with a structure of Riemann surface lamination (see \cite{Ghys_Laminations}). We say that it is a \emph{hyperbolic surface lamination} if the universal cover of every leaf is conformally equivalent to a disc. Using Candel's theorem \cite{Candel} this is equivalent to the existence of a \emph{leafwise Riemannian metric} which varies transversally continuously in local charts, such that leaves have Gaussian curvature $-1$ at every point. Recall that $\cL$ is said to be \emph{minimal} when all of its leaves are dense. 

Next we give a topological obstruction for a surface to be the leaf of a compact minimal hyperbolic surface lamination without a simply connected leaf. 

\begin{proposition}\label{p.obstpant}
Let $\cL$ be a minimal lamination by hyperbolic whose generic leaf is not a disc. Then isolated ends of leaves of $\cL$ are accumulated by genus, i.e. leaves satisfy condition $(\ast)$. Moreover, if there exists a leaf with genus and without holonomy, then every end of every leaf of $\cL$ is accumulated by genus.
\end{proposition}

\begin{proof} 
We present a slight variation of a proof that appears in \cite{ADMV}. Since the generic leaf is a hyperbolic surface with trivial holonomy (that was proved independently by Epstein-Millett-Tischler \cite{Epstein_Millett_Tischler} and by Hector \cite{Hector}) and which is not a disc by hypothesis, it contains a simple closed geodesic without holonomy $\gamma$.

Using Reeb's stability theorem \ref{t.Reeb}, the transverse continuity of Candel's hyperbolic metric and the persistence of closed geodesics under perturbations of hyperbolic metrics, we show that there exists a neighbourhood $U$ of $\gamma$ where $\cL$ induces a trivial lamination by annuli, each one of them containing a simple closed geodesic.

Assume that a leaf $L$ possesses an isolated end $e$. Since $\cL$ is minimal there exists a sequence $(x_n)_{n\in\N}$ in $L$ representing $e$ such that $x_n\in U$ for every $n$. Therefore there exists a sequence $(\gamma_n)_{n\in\N}$ of disjoint simple closed geodesics inside $L$ which converges to $e$. This implies that $e$, which is isolated, is not represented by a decreasing sequence of annuli, so it must be accumulated by genus. 

On the other hand, notice that an end $e$ is accumulated by genus if and only if there exists a sequence of simple closed geodesics $\gamma_n^1,\gamma_n^2$ in $L$ such that:
\begin{itemize}
\item $\gamma_n^i$ converges to $e$ for $i=1,2$
\item $\gamma_n^1$ and $\gamma_n^2$ intersect in exactly one point.
\end{itemize}
In this case we obtain the desired handles taking tubular neighbourhoods of $\gamma_n^1\cup\gamma_n^2$. Now, suppose the existence of a leaf without holonomy and with genus, so it contains a pair of simple closed geodesics $\gamma^1,\gamma^2$ intersecting in exactly one point. Applying again Reeb's stability and closed geodesics under perturbations of hyperbolic metrics we deduce that every end of every leaf is accumulated by pairs of closed geodesics cutting exactly once, as desired. 
\end{proof}

We show below that this obstruction is the only one and that it is possible to realize simultaneously all surfaces satisfying condition $(\ast)$ in a minimal lamination by hyperbolic surfaces whose generic leaf is a Cantor tree. This lamination will be constructed as the inverse limit of a carefully chosen tower of finite coverings of a genus $2$ surface.

\paragraph{\textbf{Towers of coverings and laminations }} Define a \emph{tower of finite coverings over a hyperbolic surface} as a sequence $\T = \{p_n : \Sigma_{n+1} \to \Sigma_n \}$ where $\Sigma_0$ is a closed hyperbolic surface and each $p_n$ is an isometric finite covering. We define the \emph{inverse limit} of $\T$ as the set $$\mathcal{L}=\left\{\x=(x_n)_{n\in\N}\in\prod_n \Sigma_n:p_n(x_{n+1})=x_n\text{ for every }n\in\N\right\}.$$  

Since $\prod_n \Sigma_n$ is a product of compact spaces and $\mathcal{L}$ is defined by closed conditions it inherits a topology that makes it a compact and metrizable topological space. In order to define the lamination structure on $\cL$ consider $\Pi_0:\cL\to\Sigma_0$ the projection on the $0$-coordinate, $\{D_{i}\}_{i\leq m}$ a finite cover of $\Sigma_0$ by open discs and $\{U_i\}_{i\leq m}$  its associated covering of $\cL$ where $U_{i}:=\Pi_0^{-1}(D_i)$. It is not difficult to see that there exist homeomorphisms $\varphi_i:U_i\to D_i\times K_i$ with $K_i$ a Cantor set and that they satisfy the compatibility conditions
$$\fhi_j\circ\fhi_i^{-1}(z,t)=(\zeta_{ij}(z,t),\tau_{ij}(t)),$$
for $(z,t)\in\fhi_i(U_i\cap U_j)$, where $\zeta_{ij}$ is holomorphic in $z$ and $\tau_{ij}$ is a homeomorphism of the Cantor set. The connected components of $\cL$ are called the leaves and are naturally endowed with structures of Riemann surfaces (see \cite{Sibony_etc} for more details). In particular, $\Pi_0$ is a Cantor-bundle over $\Sigma_0$ and its restriction of to any leaf of $\cL$ defines an isometric covering of $\Sigma_0$. 
 
 A lamination obtained by an inverse limit of coverings is always minimal (see \cite{ABMP}).

\paragraph{\textbf{Holonomy representation}} As we mentioned before, the restriction of $\Pi_0$ to each leaf $L$ of $\cL$ induces a covering map onto $\Sigma_0$. So if we choose a point $x_0\in\Sigma_0$, a preimage by the projection $x\in L$ and a closed path $c$ based at $x_0$ there is a unique lift of $c$ to $L$ starting at $x$. Its endpoint only depends on the homotopy class $\gamma\in\pi_1(\Sigma_0)$ of $c$ and is denoted by $\tau_\gamma(x)$. 

The map $\tau_\gamma$ is an homeomorphism of the fiber $K$ of $x_0$ (which is a Cantor set) and the correspondence $\fhi:\pi_1(\Sigma_0)\to\Homeo(K), \gamma\mapsto\tau_\gamma^{-1}$ defines a group morphism. In the sequel this morphism will be called the \emph{holonomy representation of $\cL$}.

\paragraph{\textbf{Laminated bundles and supension}} Reciprocally any $\fhi:\pi_1(\Sigma_0)\to\Homeo(K)$ is the holonomy representation of a laminated Cantor-bundle over $\Sigma_0$ obtained by a process called \emph{suspension}. See for example \cite{Camacho_LinsNeto,candel-conlon} for a detailed treatment.

Let $\widetilde{\Sigma}_0$ be the universal covering of $\Sigma_0$ and $\rho$ be the action of $\pi_1(\Sigma_0)$ on $\widetilde{\Sigma}_0$ by deck transformations. The product $\rho\times\fhi$ defines the \emph{diagonal action} of $\pi_1(\Sigma_0)$ on $\widetilde{\Sigma}_0\times K$ which is properly discontinuous. The quotient of this action is denoted by $\cL$ and is called the \emph{suspension} of $\fhi$. 

The projection on the first coordinate descends to a fiber bundle $\Pi:\cL\to\Sigma_0$ with fiber $K$. Moreover the partition $(\widetilde{\Sigma}_0\times\{x\})_{x\in K}$ passes to the quotient and provides $\cL$ with a structure of lamination. With say that this lamination is \emph{transverse} to the bundle given by $\Pi$. Finally the holonomy representation of this lamination is given by $\fhi$.

Furthermore, if two laminated Cantor-bundles $\Pi:\cL\to\Sigma_0$ and $\Pi':\cL'\to\Sigma_0$, have the same holonomy representation they are \emph{equivalent} in the sense that there exists a homeomorphism $H:\cL\to\cL'$ satisfying $\Pi'\circ H=\Pi$ (so in particular $H$ preserves fibers) and taking leaves of $\cL$ onto leaves of $\cL'$.

Finally the holonomy representation of a lamination encodes all its dynamics. In particular, the lamination $\cL$ is minimal if and only if the action on $K$ given by its holonomy representation is minimal (i.e. all the orbits are dense). We refer to \cite[Chapter V]{Camacho_LinsNeto} for all these facts.

\section{From graphs to surfaces}\label{s.graphs}
In this section we translate Theorem \ref{maintheorouille} into an analogous theorem in the context of graphs (see Theorem \ref{t.grafos}). The second subsection is devoted to an outline of the strategy for the proof of this theorem.

\subsection{Graphs and laminations} A graph $\Gamma$ consists of
 a set of vertices $V(\Gamma)$ together with a set of edges $E(\Gamma)$ contained in $V(\Gamma)\times V(\Gamma)$. We will most of the time identify graphs with their topological realizations. We say that a map $f:\Gamma_1\to \Gamma_2$ between graphs is a \emph{graph morphism} if it is a continuous map preserving vertices and sending each edge either to a vertex or to an edge.

\paragraph{\textbf{Classifying triples and condition $(\ast)$ }} Let $\Gamma$ be a non-compact, locally finite graph and $(K_n)_{n\in\N}$ be an increasing and exhausting sequence of finite subgraphs. Recall that an end $e\in\cE(\Gamma)$ is a decreasing sequence $(\cC_n)_{n\in\N}$ of connected components of $\Gamma\moins K_n$. We say that the end $e$ is \emph{accumulated by homology} if $H_1(\cC_n,\R)\neq 0$ for every $n$. Notice that the definition does not depend on the choice of the exhausting sequence. We denote by $\cE_0(\Gamma)$ the set of ends accumulated by homology which is a closed subset of $\cE(\Gamma)$.  

We define the \emph{classifying triple} of a non-compact and locally finite graph $\Gamma$ as $\tau(\Gamma)=(g(\Gamma),\cE_0(\Gamma),\cE(\Gamma))$ where $\cE_0(\Gamma)$ and $\cE(\Gamma)$ have been defined above and where
$$g(\Gamma):=\beta_1(\Gamma)=\dim H_1(\Gamma,\R),$$
is the \emph{first Betti number} of $\Gamma$. Finally we say that a graph $\Gamma$ \emph{satisfies condition $(\ast)$} if its classifying triple $\tau(\Gamma)$ does so. This is equivalent to say that the isolated points of $\cE(\Gamma)$ belong to $\cE_0(\Gamma)$ (or that the isolated ends are accumulated by homology).

Notice that classifying triples are no longer complete invariants in the context of graphs (i.e. there is no analog of Theorem \ref{classification}). For example, we can imagine graphs with different systoles but with equivalent classifying triples. However, in \S \ref{ss.fromgtos} we formalize the procedure of thickening a graph to obtain a surface with equivalent classifying triple justifying our terminology (see Proposition \ref{p.from_graphs_surfaces}). Finally, note that we define condition $(\ast)$ for graphs so that the corresponding thickened surface also satisfies condition $(\ast)$.

\paragraph{\textbf{Laminations by graphs and towers }} Consider a tower of finite coverings $\U=\{q_n:\Gamma_{n+1}\to\Gamma_n\}$ over a finite graph $\Gamma_0$. As in the surface case, we can define $\cM$ as the inverse limit of the tower. For the same reason as in the surface case, $\cM$ is a compact and metrizable topological space, but in this case it is locally a product of a Cantor set by a graph, we call such a structure: a \emph{lamination by graphs}. As in the surface case, leaves correspond to connected components of $\cM$ and the restriction of the projection $\Pi_0$ to any leaf  of $\cM$ is a covering of $\Gamma_0$. 

Again as in the surface case, we say that $\cM$ is transverse to a Cantor-fiber bundle. It is possible to generalize the discussion on laminated bundles to this context, in particular we use that $\cM$ determines a holonomy representation $\fhi:\pi_1(\Gamma_0)\to\Homeo(K)$ and that such representation uniquely determines (up to equivalence) a Cantor-fiber bundle laminated by graphs via the suspension process.

\paragraph{\textbf{Generic leaf }} The following result (whose proof will be omitted) is analogous to the second part of Proposition \ref{p.obstpant} but in the context of graphs. We use it to guarantee that the generic leaf of $\cM$ is a tree.

\begin{proposition}\label{p.generic_tree} Consider $\mathbb{U}=\{q_n:\Gamma_{n+1}\to\Gamma_n\}$, a tower of finite coverings of finite graphs and $\cM$ its inverse limit. If $\cM$ contains a leaf which with finite dimensional homology, then the generic leaf is a tree. 
\end{proposition}

\subsection{From graphs to surfaces}\label{ss.fromgtos}
In this section we show how to produce surface laminations from graph laminations. Intuitively we thicken the graphs and then remove the interior of the thickening. 
\paragraph{\textbf{Pinching maps }} Given a graph $\Gamma$, we say that $e\subseteq\Gamma$ is an \emph{open edge} if it is a connected components of $\Gamma\setminus V(\Gamma)$. We say that a map $f:S\to \Gamma$ between a surface $S$ and a graph $\Gamma$ is a \emph{pinching map} if it satisfies the two following properties
\begin{itemize}
\item $f^{-1}(v)$ is homeomorphic to a $n$-holed sphere with boundary for every vertex of $\Gamma$, where $n$ is the valency of $v$.
\item  $f^{-1}(e)$ is an open cylinder for every open edge $e\subseteq\Gamma$.
\end{itemize}
It is straightforward to check that if $S$ and $\Gamma$ are non-compact and $f:S\to \Gamma$ is a pinching map then, $S$ and $\Gamma$ have equivalent classifying triples.

\begin{figure}[h!]
\centering
\includegraphics[scale=0.13]{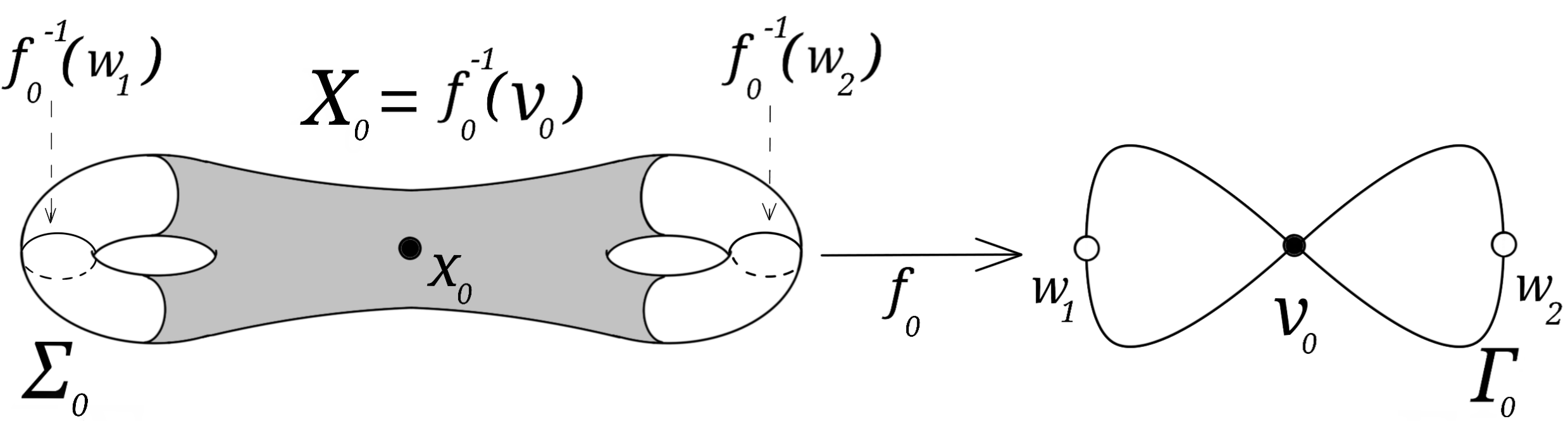}
\caption{A pinching map.}\label{fig:pinchouille}
\end{figure}

Consider now $f_0:(\Sigma_0,x_0)\to(\Gamma_0,v_0)$ a pinching map between a closed surface of genus $2$ and the wedge of two circles at the vertex $v_0$, that we call the figure eight graph (see Figure \ref{fig:pinchouille}). Note that $\pi_1(\Gamma_0,v_0)$ is a free group on two generators, denoted by $\F_2$. Denote $H=\Ker(f_0)_\ast\subseteq \pi_1(\Sigma_0,x_0)$ and $\widehat{\Sigma}_0$ the covering space of $\Sigma_0$ associated to $H$. Finally let $\widetilde\Gamma_0$ denote the universal cover of $\Gamma_0$ and $\rho_\ast$, $\rho$ denote the $\mathbb{F}_2$-actions by deck transformations in $\widehat\Sigma_0$ and $\widetilde\Gamma_0$ respectively. 

Note that, $f_0$ lifts to an $\mathbb{F}_2$-equivariant pinching map $$\overline f_0:\widehat\Sigma_0\to \widetilde\Gamma_0$$  (here we identify $\pi_1(\Sigma_0,x_0)/H$ with $\pi_1(\Gamma_0,v_0)\simeq\F_2$.)

\paragraph{\textbf{From graph laminations to surface laminations}} Let $\mathbb{U}=\{q_n:\Gamma_{n+1}\to\Gamma_n\}$ be a tower of finite coverings over $\Gamma_0$ and $\mathcal{M}$ be its inverse limit. Recall that the projection on the $0$-coordinate $\Pi_0:\mathcal{M}\to \Gamma_0$ is a Cantor-bundle. We let $\fhi:\mathbb{F}_2\to \text{Homeo}(K)$ denote holonomy representation so $\cM$ is equivalent to the suspension of $\fhi$, that is, as we recall, the quotient of $\widetilde\Gamma_0\times K$ under the diagonal action given by $\phi=\rho\times\fhi$. Then, the following holds:

\begin{proposition}[From graphs to surfaces]\label{p.from_graphs_surfaces}
There exists a surface-laminated Cantor-bundle $\Pi:\cL\to\Sigma_0$ and a continuous map $\psi:\cL\to\cM$ satisfying
\begin{itemize}
\item $\Pi\circ\psi=\Pi_\ast$ and $\psi$ is a fiberwise homeomorphism;
\item $\psi$ sends leaves of $\cL$ to leaves of $\cM$, and induces a bijection between the corresponding sets of leaves;
\item $\psi$ conjugates the holonomy representations of $\cL$ and $\cM$; in particular $\cL$ is minimal;
\item the restriction of $\psi$ to every leaf of $\cL$ is a pinching map, so it preserves classifying triples.
\end{itemize}
\end{proposition} 

\begin{proof} We define $\cL$ as the quotient of $\widehat\Sigma_0\times K$ under the diagonal action defined by $\phi_\ast=\rho_\ast\times \fhi$. First notice that $\cL$ is the suspension of the representation $\fhi\circ(f_0)_\ast$ so it is a laminated Cantor-fiber bundle over $\Sigma_0$. In particular $\cL$ is a compact lamination. Also, since $\mathcal{M}$ is minimal, so is the action $\fhi$ and consequently, so is $\cL$.

Consider now the map $F=\overline{f}_0\times \Id$ and notice that it is $(\phi_\ast,\phi)$-equivariant. Therefore, $F$ descends to a continuous map $\psi:\cL\to\mathcal{M}$. By construction $\psi$ induces a fiberwise homeomorphism and conjugates the holonomy representations of $\cL$
 and $\cM$, in particular it induces a bijection between the corresponding sets of leaves. Finally, by definition of $\overline{f}_0$, the restriction of $\psi$ to each leaf of $\cL$ is a pinching map onto its image.
\end{proof}

Therefore, the proof of Theorem \ref{maintheorouille} reduces to that of:
\begin{theorem}\label{t.grafos} There exists a tower of finite coverings $\U=\{q_n:\Gamma_{n+1}\to\Gamma_n\}$ over the figure eight graph $\Gamma_0$, whose inverse limit $\cM$ satisfies
\begin{enumerate}
\item\label{condi1} its generic leaf is a tree;
\item\label{condi2} given any classifying triple $\tau$ satisfying condition $(\ast)$, there exists a leaf of $\cM$ whose classifying triple is equivalent to $\tau$.
\end{enumerate}
\end{theorem}

\subsection{Strategy of the proof of Theorem \ref{t.grafos}}\label{ss.strategy} 
Consider a tower of finite coverings  of finite graphs $\mathbb{U}=\{q_n:\Gamma_{n+1}\to\Gamma_n\}$ and denote its inverse limit by $\cM$. In order to obtain Theorem \ref{t.grafos} we must answer several questions. How to recognize the topological type of a leaf of $\cM$? How to construct a leaf with prescribed classifying triple? How to make sure that all classifying triples satisfying condition $(\ast)$ are realized by leaves of $\cM$? 

\paragraph{\textbf{Recognizing the topology of a leaf}} The first tool we need in order to study the topology of the leaves is the concept of \emph{direct limit} of sequences of graph inclusions (see the definition given in \S \ref{ss.limits_forest}). More precisely assume that there exist
\begin{itemize}
\item a sequence of subgraphs $G_n\subseteq\Gamma_n$, and 
\item a sequence of $(1:1)$-lifts $j_n:G_n\to G_{n+1}$ satisfying $j_n(G_n)\subseteq \Int(G_{n+1})$.
\end{itemize}
We then say that the chain of inclusions $\{j_n:G_n\to G_{n+1}\}$ in \emph{included inside} the tower. We can prove (this is done in Proposition \ref{p.iso}) that in this case the direct limit of the chain $\{j_n:G_n\to G_{n+1}\}$ is isomorphic to a leaf of $\mathcal{M}$. 

So our strategy consists in realizing all classifying triples satisfying condition $(\ast)$ in graphs obtained as direct limits of chains $\{j_n:G_n\to G_{n+1}\}$, and then to include such chains inside a tower of finite coverings of the figure eight graph, as defined above.

\paragraph{\textbf{Constructing a leaf with prescribed classifying triple}} To include such a chain inside a tower of finite coverings requires to control the topology of graphs $G_n$ so as to prescribe that of the direct limit. There are combinatorial constraints to do so, and it would be too tedious to know exactly which chain can be included inside a tower, and how to perform the inclusion of a given chain. This is one of the reasons for providing our graphs with a \emph{decoration}, i.e. associate different types to the vertices. We define \emph{$C$-graphs}, which represent usual graphs up to some local information that does affect the \emph{asymptotic invariants} we are interested in (ends, ends accumulated by homology, etc.). Therefore, prescribing a chain of $C$-graphs is equivalent to prescribing a chain of usual graphs up to some local invariants, which will make much easier its inclusion inside a tower of finite coverings. These $C$-graphs are related to classical graphs via an operation of \emph{collapsing} which takes some finite connected subgraphs with homology into a new type of vertices called $h$-vertices. 

The tower is built inductively. And the induction step, i.e. the  construction of the covering $q_n:\Gamma_{n+1}\to\Gamma_n$ requires to stabilize the topology of the subgraph $G_n$ (this graph must be lifted to $\Gamma_{n+1}$). This is done by an operation of \emph{surgery} of coverings. The formalism of $C$-graphs is also well suited for this surgery operation.

\paragraph{\textbf{Realizing all classifying triples}} In order to realize all the classifying triples, we shall include several chains of $C$-graphs inclusions inside a tower of finite coverings. We can see such a chain $\{j_n:G_n\to G_{n+1}\}$ as a ray in some arborescent structure called a \emph{forest} (a disjoint union of trees), whose ends will provide infinite graphs with the desired classifying triples.

So in order to realize all classifying triples simultaneously, we have to generalize the concept of inclusion of a single chain inside a tower, to that of \emph{the inclusion of a whole forest of graphs inside a tower of finite coverings}. This is done in \S \ref{ss.forests_inclusion_realization}. 

Finally, we have to make sure that the ends of those forests that we construct represent (almost) all possible classifying triples of infinite graphs: this is the purpose of Section \ref{s.forestuniv}. Actually, we will see that our formalism of forest and $C$-graphs forces us to treat separately the case of infinite  graphs with finite dimensional homology, and that of graphs with infinite dimensional homology. Finally, the fact that generic leaves of the lamination are trees will be deduced from Proposition \ref{p.generic_tree} and the fact that the constructed laminations contains leaves with finite dimensional homology.

\section{$C$-graphs}\label{s.decorouille}

As we explained above, it will be convenient to decorate our graphs in order to perform the two key operations in the proof of Theorem \ref{t.grafos}: \emph{collapses and surgeries}. We develop in this section the formal framework of $C$-graphs. These graphs possess special vertices called $h$-vertices which represent finite and connected subgraphs with positive first Betti number which are related to the operation of collapse as we shall see. There is another type of vertices, called boundary vertices, that will be useful in the treatment of subgraphs and surgeries. There are also some restrictions on the valencies of different types of vertices of $C$-graphs whose necessity will become clear in Section \ref{s.surgeries}.

\subsection{Definition of $C$-graphs}\label{s.def_cgraph}
We say that $\Gamma$ is a $C$-\emph{graph} ($C$ stands for \emph{collapse}) if it is a graph with three types of vertices:
\begin{itemize}
\item \emph{boundary vertices}, that may have valency $1$ or $2$;
\item \emph{simple vertices}, that have valency $4$; and 
\item \emph{homology vertices}, that may have valency $2$ or $4$. 
\end{itemize}
Moreover, we ask edges to join vertices of different types one of them being a boundary vertex. We refer to these vertices as $b,s$ or $h$-vertices. Sometimes we add an index to specify their valencies, that is we are going to have $b_1,b_2,s_4,h_2$ and $h_4$-vertices. 

In Figure \ref{fig:decoratouille} we see the figure eight $(\Gamma,v)$ with one $s$-vertex, two $b_2$-vertices and four edges. We call this graph, the \emph{figure eight $C$-graph}. Consider $$\mathbb{U}=\{q_n:\Gamma_{n+1}\to\Gamma_n\}$$ a tower of finite coverings with $\Gamma_0$ a finite covering space of the figure eight $C$-graph. Then, we say that $\mathbb{U}$ is a \emph{tower}.

\begin{figure}[h!]
\centering
\includegraphics[scale=0.05]{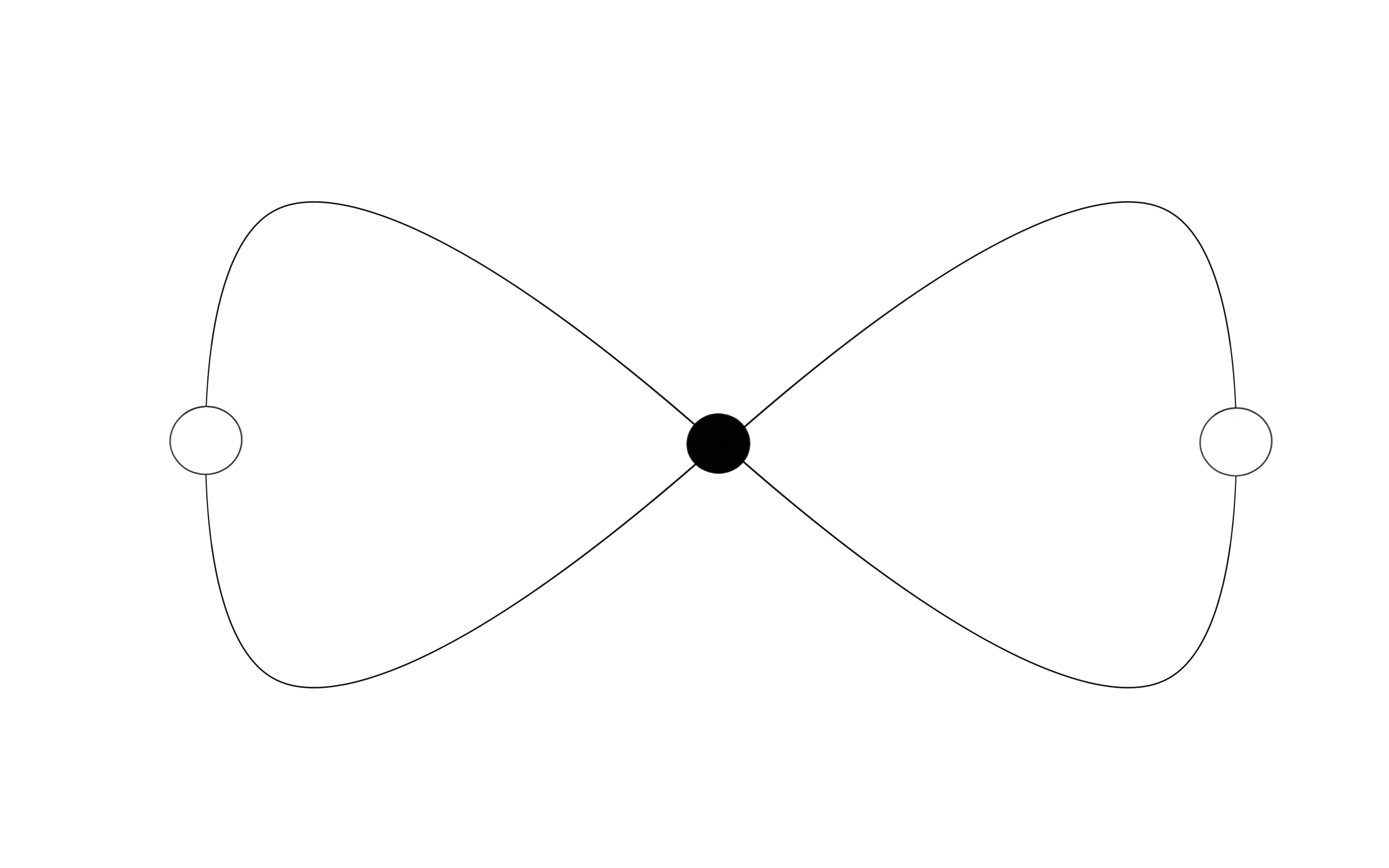}
\caption{The  figure eight $C$-graph with one $s$-vertex and two $b_2$-vertices. In that paper, all $s$-vertices will be represented in black and all $b$-vertices will be represented in white.}\label{fig:decoratouille}
\end{figure}

Given a $C$-graph $G$, denote $\text{dist}_{G}$ the path metric in $G$ where all the edges have length $1$ and $B_G(v,r):=\{x\in G:\text{dist}_G(v,x)\leq r\}$. We define the \emph{boundary} of $\Gamma$ as the set $$\partial\Gamma:=\{v\in\Gamma:v\text{ is a vertex or type } b_1\}$$ and the \emph{interior} of $\Gamma$ as $\text{Int}(\Gamma):=\Gamma\setminus\partial\Gamma$. 

It is practical for the construction of Theorem \ref{t.grafos} that vertices in the topological boundaries of subgraphs have valency-$1$, both in the subgraph and in its complement. This is the reason for introducing boundary vertices and for the next definition: we say that a subgraph $S$ of $G$ is a \emph{$C$-subgraph} if for each vertex $v\in S$ which is not of boundary type, it holds that $B_G(v,1)\subseteq S$. Note that $C$-subgraphs are naturally $C$-graphs.

\subsection{$C$-graphs and ends spaces.} In that paper, we think $h$-vertices as vertices with non-trivial homology. But we do not assign a particular Betti number to such vertices. Therefore $C$-graphs with finite-dimensional homology and finitely many $h$-vertices have undetermined Betti number. For this reason, when working with $C$-graphs with $h$-vertices we use ends pairs instead of classifying triples. 

We define the ends space of a $C$-graph as the ends space of its underlying graph (recall that $h$-vertices represent finite and connected subgraphs). On the other hand, since $h$-vertices represent subgraphs with positive Betti number we say that a vertex $\alpha\in\cE(G)$, represented by a decreasing sequence of subgraphs $(\mathcal{C}_n)_{n\in\N}$, belongs to $\cE_0(G)$ if every $\mathcal{C}_i$ either has non-trivial homology or contains a vertex of $h$-type. We call $\cE_0(G)$ the \emph{space of ends of $G$ accumulated by homology}.

It worth mentioning that $\cE_0(G)$ is closed in $\cE(G)$ and its definition does not depend on the choice of the sequence $(\mathcal{C}_i)_{i\in\N}$. Given a $C$-graph $G$ we define its \emph{pair of ends} as the pair $(\cE_0(G),\cE(G))$. Finally, we say that two pairs of ends $(\cE_0(G),\cE(G))$ and $(\cE_0(H),\cE(H))$ are \emph{equivalent} if there exists an homeomorphism $h:\cE(G)\to\cE(H)$ satisfying $h(\cE_0(G))=\cE_0(H)$. 

\begin{remark}\label{r.star} Note that if a $C$-graph $\Gamma$ has an isolated end $\alpha$ which is not accumulated by homology in the classical sense (i.e. it contains a neighbourhood in $\Gamma$ which is topologically a tree) then, it must have a neighbourhood containing only vertices of type $b_2$ and $h_2$. Therefore, ends pairs of $C$-graphs always satisfy condition $(\ast)$. 
\end{remark}

\paragraph{\textbf{Examples of $C$-graphs}} The next two pictures illustrate important examples of $C$-graphs that have finitely many ends and that we will use during the proof of Theorem \ref{t.grafos}. Each end of these graphs is accumulated by vertices of $h$-type. We first illustrate $C$-graphs with an even number of ends in Figure \ref{fig:evouille} below. They are topologically trees and all ends are accumulated by vertices of type $h_2$ (in all figures, the subscripts of $h$-type vertices will be omitted, as it only reflects their valencies, which will be clear from the pictures): this illustrates Remark \ref{r.star}.

\begin{figure}[h!]
\centering
\includegraphics[scale=0.11]{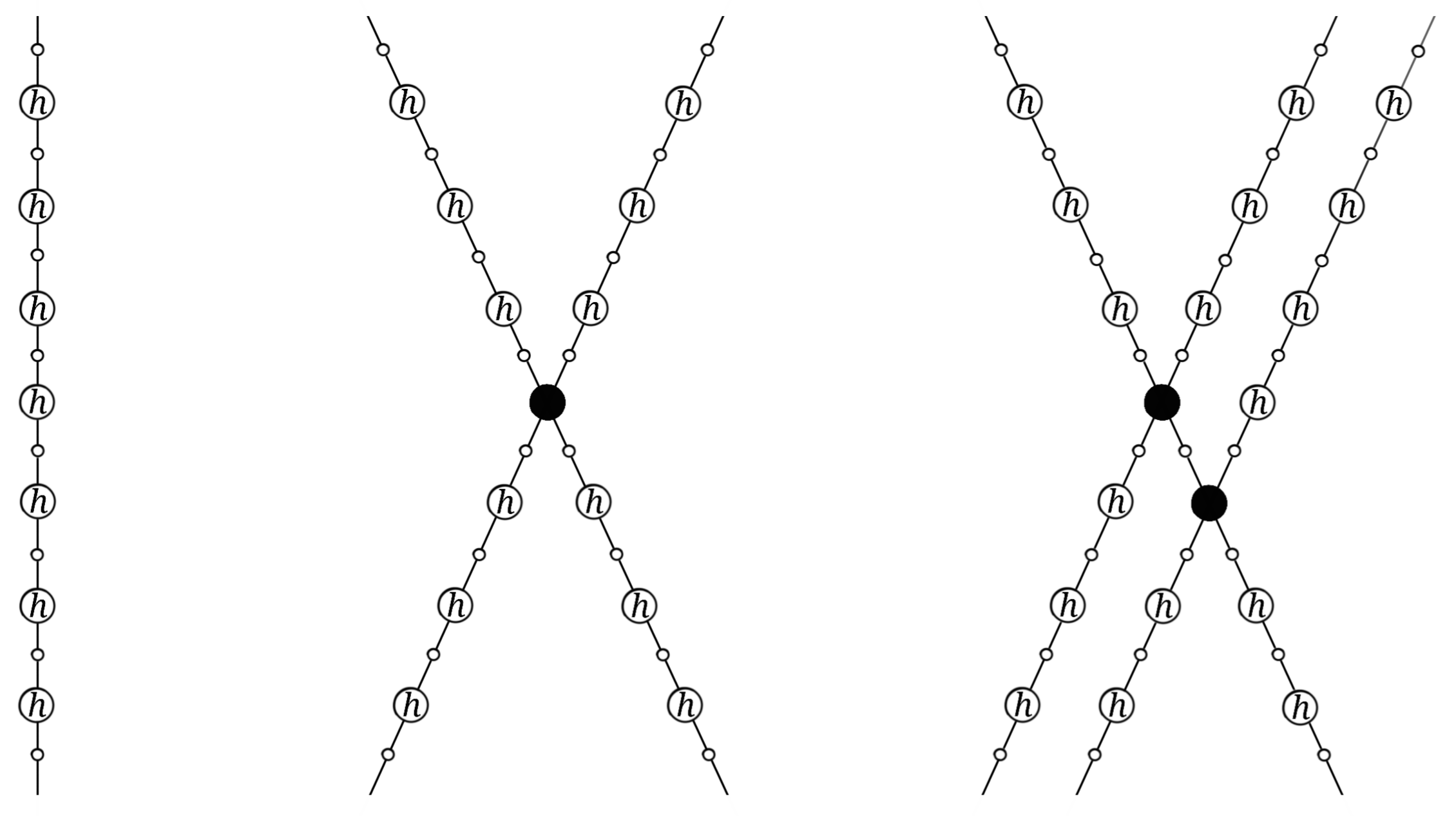}
\caption{$C$ graphs with respectively $2$, $4$ and $6$ ends. Proceeding as suggested in the figure, one obtains $C$-graphs with all even numbers of ends.}\label{fig:evouille}
\end{figure}

Note that if a $C$-graph with finitely many ends is topologically a tree then, since the valency of each vertex equals $2$ or $4$, the number of ends must be even. Hence, to realize $C$-graphs with an odd number of ends, we need to use vertices of type $h_4$, as pictured in Figure \ref{fig:forouille}. Representing $C$-graphs with an odd number of ends is the only reason why we consider vertices of type $h_4$.

\begin{figure}[h!]
\centering
\includegraphics[scale=0.11]{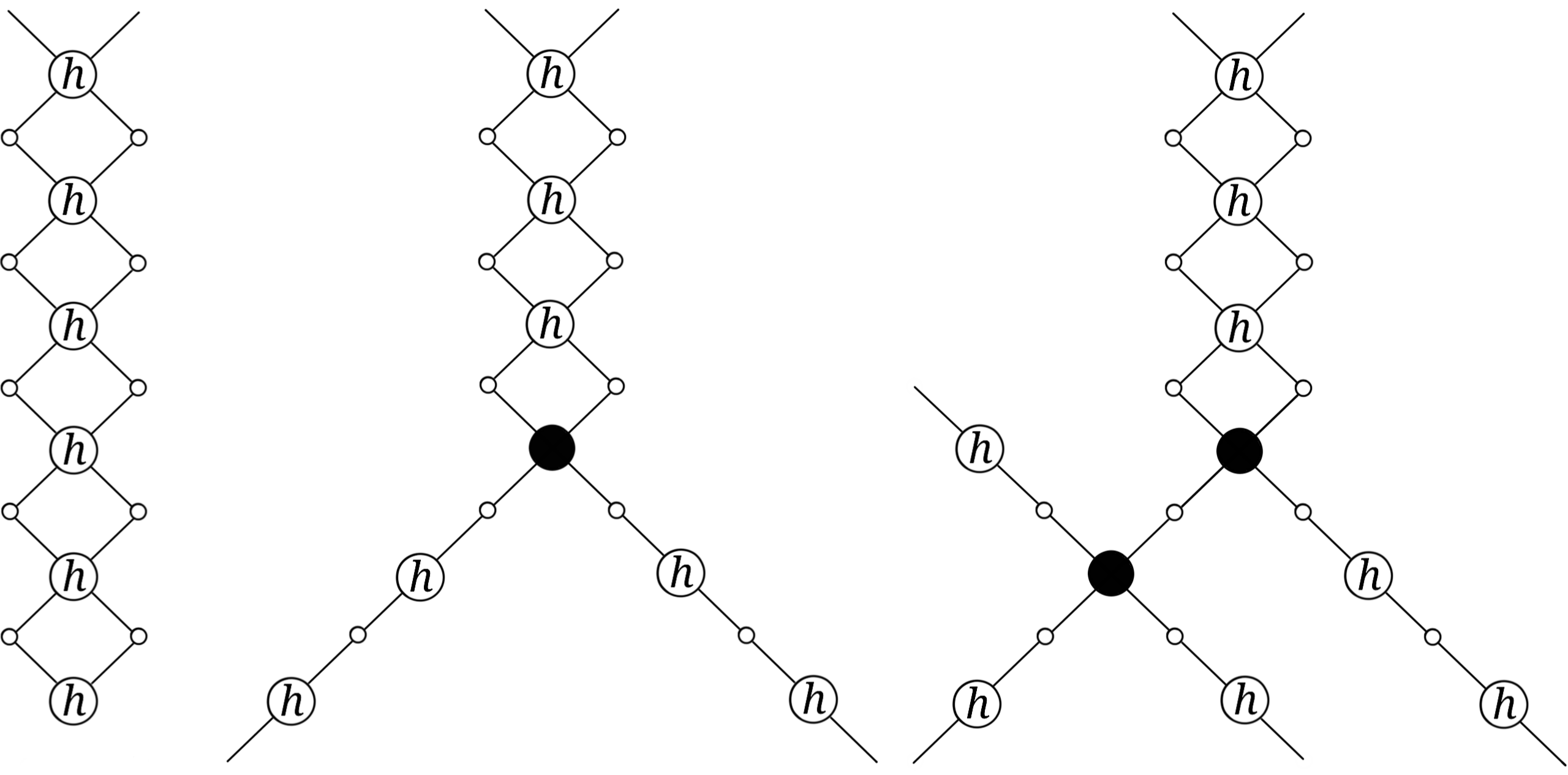}
\caption{$C$-graphs with respectively 1,3 and 5 ends. Proceeding as suggested in the figure, one obtains $C$-graphs with all odd numbers of ends. Any such graph has a special end accumulated by $h_4$-vertices.}\label{fig:forouille}
\end{figure}

\subsection{Collapses.}\label{ss.collapses}
If $G$ is a graph and $\mathcal{F}$ a countable family of finite and connected subgraphs of $G$, let $G/{\mathcal{F}}$ denote the quotient of $G$ under the equivalence relation ``being on the same subgraph of $\mathcal{F}$''. Notice that $G/{\mathcal{F}}$ has a natural graph structure. 

Consider $C$-graphs $G$ and $H$. Assume that $G$ contains no $h$-vertices and denote $H_{\ast}$ the set of $h$-vertices of $H$. We say that  a map $f:G\to H$ is a \emph{collapse} if it is a graph morphism and there exist: \begin{itemize}
\item $\mathcal{F}$ a disjoint family of finite, connected and homologically non-trivial subgraphs of $G$ contained in $\text{Int}(G)$ and
\item $\hat f:G/{\mathcal{F}}\to H$ a graph isomorphism such that:
\begin{enumerate}
\item $\hat f\circ \pi=f$ where $\pi:G\to G/{\mathcal{F}}$ is the quotient map;
\item\label{a1}$f$ induces a bijection between $\mathcal{F}$ and $H_{\ast}$;
\item\label{a2} $f$ preserve vertex types when restricted to $G\setminus \cup_{S\in\mathcal{F}}S$.
\end{enumerate}
\end{itemize}

\begin{figure}[h!]
\centering
\includegraphics[scale=0.13]{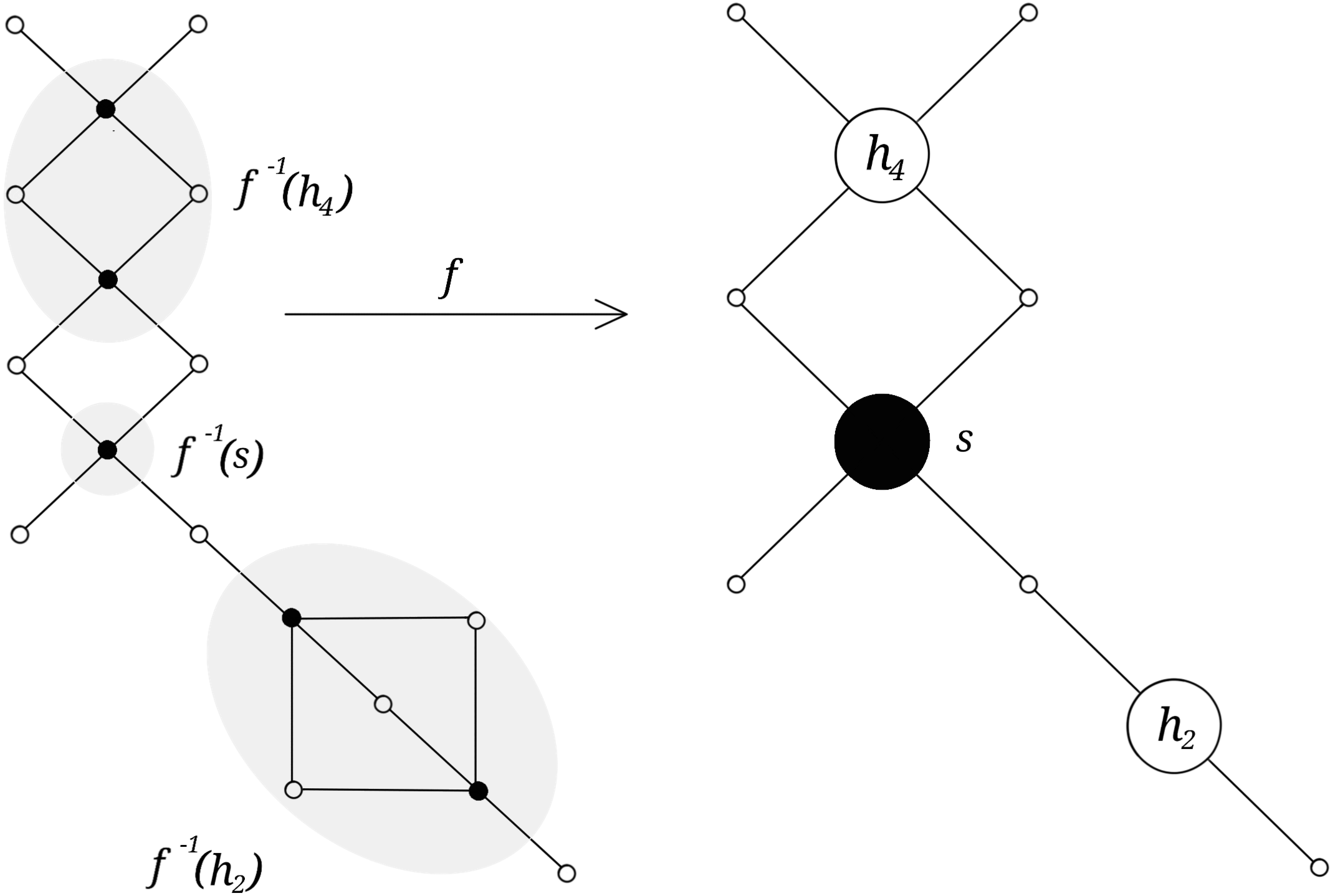}
\caption{A collapse. Each shaded subgraph of the graph on the left is collapsed onto a vertex of the graph on the right.}\label{fig:coditou}
\end{figure}  

\begin{remark}\label{r.proper} Notice that, since the subgraphs in the family $\mathcal{F}$ are finite, collapses are proper maps. 
Also, conditions \eqref{a1} and \eqref{a2} imply that preimages of $b_i$ vertices under collapses are $b_i$ vertices for any valency $i=1,2$. 
\end{remark}

\subsection{Collapses and end spaces.}\label{ss.collapses and triples} The following proposition shows that although collapses forget some topological information, they preserve ends pairs.

\begin{proposition}\label{p.codingends} Consider infinite $C$-graphs $G,H$ and a collapse $f: G\to H$. Then, the ends pairs $(\cE_0(G),\cE(G))$ and $(\cE_0(H),\cE(H))$ are equivalent.
\end{proposition}

In order to prove Proposition \ref{p.codingends} we need the following Lemma.
\begin{lemma}\label{l.connected} Consider a collapse $f:G\to H$ and a connected subgraph $H_0\subseteq H$. Then, $f^{-1}(H_0)$ is connected.
\end{lemma}
\begin{proof} Write $f^{-1}(H_0)=\bigsqcup_{n}\mathcal{B}_n$ where $\mathcal{B}_n$ are the connected components of $f^{-1}(H_0)$. Since $H_0$ is connected, there must exist $i\neq j$ such that $f(\mathcal{B}_i)\cap f(\mathcal{B}_j)\neq\emptyset$. This is absurd since by definition of collapse, pre-images of vertices are connected. 
\end{proof}

\begin{proof}[Proof of Proposition \ref{p.codingends}.] Let  $(K_n)_{n\in\N}$ be an exhaustion of $H$ by compact subgraphs. Up to modifying the sequence we can suppose that every connected component of $H\moins K_n$ is unbounded. Define $L_n:=f^{-1}(K_n)$ and notice that, since collapses are proper (see Remark \ref{r.proper}), the sequence $L_1\subseteq L_2\subseteq\ldots$ is an exhaustion of $G$ by compact subsets. Also notice that, by Lemma \ref{l.connected}, taking preimages induces a $(1:1)$ correspondence between the connected components of $H\setminus K_n$ and those of $G\setminus L_n$. 

We proceed to define the homeomorphism between $\cE(H)$ and $\cE(G)$. For this, take an end $\alpha\in\cE(H)$ defined by a decreasing sequence of subsets $(\mathcal{C}_n)_{n\in\N}$, where each $\mathcal{C}_n$ is a connected component of $H\moins K_n$. We define $\fhi(\alpha)$ as the end represented by the decreasing sequence $(\mathcal{B}_n)_{n\in\N}$ where $\mathcal{B}_n=f^{-1}(\mathcal{C}_n)$. This definition makes sense because $f^{-1}(\mathcal{C}_n)$ is a connected component of $G\setminus L_n$. It is straightforward to check that $\fhi$ is an homeomorphism. 

To prove that $\varphi(\cE_0(H))\subseteq \cE_0(G)$, take $\alpha\in\cE_0(H)$ defined by a sequence $(\mathcal{C}_n)_{n\in\N}$ and denote $\mathcal{B}_n=f^{-1}(\mathcal{C}_n)$. We need to check that $\beta_1(\mathcal{B}_n)> 0$ for every $n\in\N$. For this we distinguish two cases.

\emph{Case 1. $\beta_1(\mathcal{C}_n)\neq 0$}.\\ In this case, we can find a $b_2$-vertex $w\in\mathcal{C}_n$ such that $\mathcal{C}_n\moins\{w\}$ is connected. Since preimages of $b_2$-vertices under collapses are $b_2$-vertices, we have that $w_0:=f^{-1}(w)$ is also a $b_2$-vertex (see Remark \ref{r.proper}). On the other hand, Lemma \ref{l.connected} implies that $f^{-1}(\mathcal{C}_n\moins\{w\})$ (which equals $\mathcal{B}_n\setminus \{w_0\}$) is connected. Finally, since $w_0$ has valency $2$ and $\mathcal{B}_n\setminus \{w_0\}$ is connected we conclude that $\beta_1(\mathcal{B}_n)>0$ as desired.

\emph{Case 2. $\mathcal{C}_n$ contains a vertex $v$ of type $h$.}\\ This case follows directly from the definition of collapse. 

To finish we check that $\varphi(\cE_0(H)^c)\subseteq \cE_0(G)^c$. For this, consider an end $\alpha$ defined by a sequence $(\mathcal{C}_n)_{n\in\N}$. In this case, there must exist an integer $n_0>0$ such that $\mathcal{C}_{n_0}$ is a tree without $h$-vertices. Then, by definition of collapse we have that $\mathcal{B}_{n_0}=f^{-1}(\mathcal{C}_{n_0})$ is a tree and therefore, $\fhi(\alpha)\notin\cE_0(G)$ as desired. This finishes the proof of the Proposition.
\end{proof}

\subsection{Elementarily decomposable $C$-inclusions} According to the strategy of the proof of Theorem \ref{t.grafos} outlined in \S \ref{ss.strategy}, we need to ``realize'' some inclusions of $C$-graphs as lifts inside a tower of coverings. Elementary inclusions will be the basic blocks of those inclusions that our techniques allow us to realize. This will become clear and formal in the next two sections.

\paragraph{\textbf{Basic pieces and $C$-inclusions}}Given a $C$-subgraph $B$, we say it is a \emph{basic piece} if $B=B_G(v,1)$ where $v$ is a non-boundary vertex. We call them $s$, $h_2$ or $h_4$-pieces according to the vertex type of $v$. Notice that basic pieces are the smallest possible $C$-subgraphs. 

\begin{figure}[h!]
\centering
\includegraphics[scale=0.04]{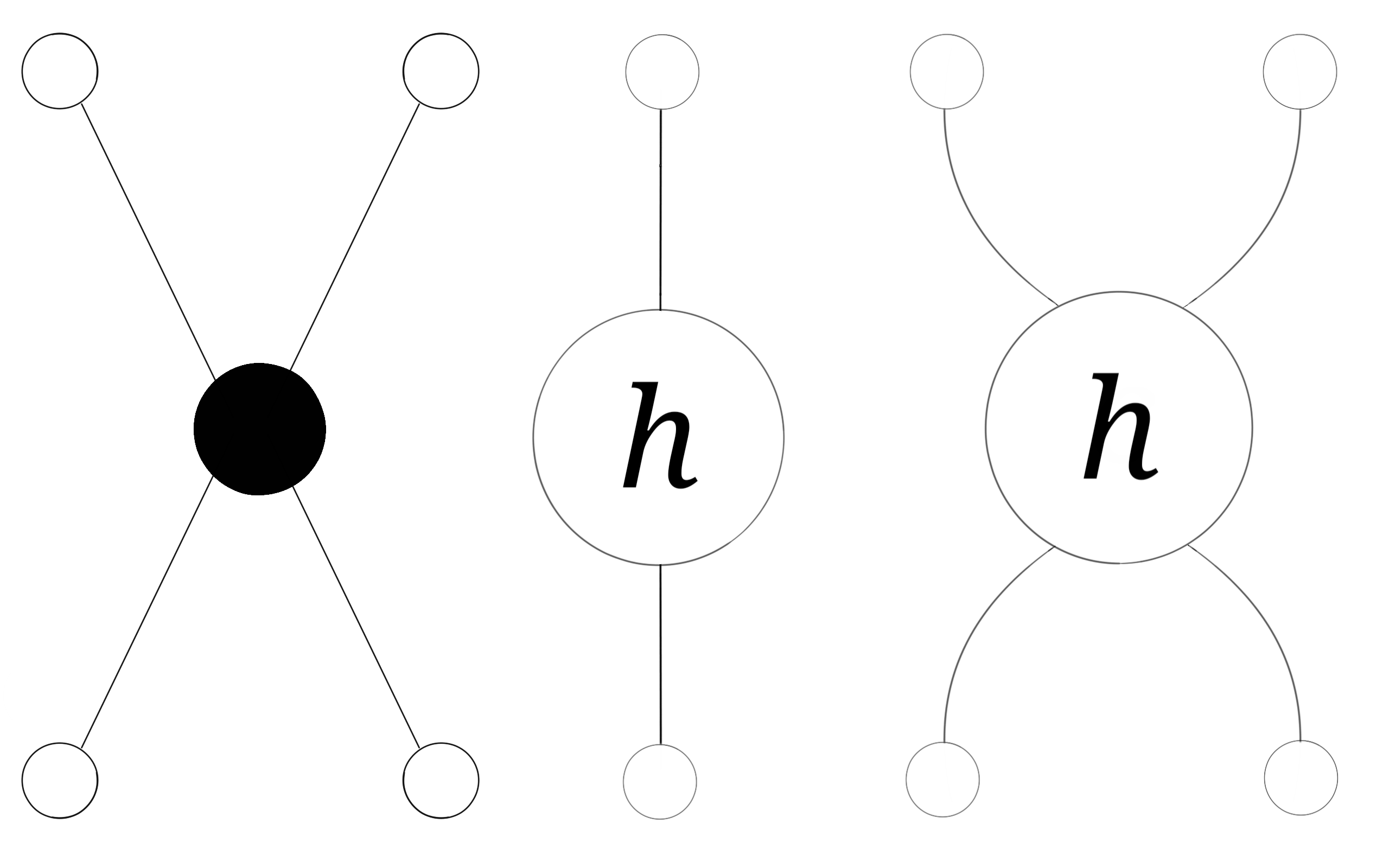}
\caption{Basic pieces.}\label{fig:basicou}
\end{figure}  

An injective map between $C$-graphs $\iota:G_1\to G_2$ is a \emph{$C$-inclusion} if it is injective, preserve vertex types and satisfies that $\iota(G_1)$ is a $C$-subgraph of $G_2$. Notice that it may happen that the $C$-inclusion sends a $b_1$-vertex of $G_1$ into a $b_2$-vertex of $G_2$. 

\paragraph{\textbf{Elementarily decomposable $C$-inclusions }}

We say that a $C$-inclusion $\iota:H_1\to H_2$ is \emph{elementary} if it satisfies that $H_2=\iota(H_1)\cup B$ where $B$ can be
\begin{enumerate}
\item an $h_2$-piece meeting $\iota(H_1)$ in a single boundary vertex;
\item an $s$-piece meeting $\iota(H_1)$ in a single boundary vertex, or
\item an $h_4$-piece meeting $\iota(H_1)$ in exactly two boundary vertices.
\end{enumerate}
Also, we say that $\iota:H_1\to H_2$  is an \emph{elementarily decomposable $C$-inclusion} if there exist:
\begin{itemize}
\item  $C$-graphs $K_1,\ldots,K_n$ with $K_1=H_1$ and $K_n=H_2$ and 
\item elementary $C$-inclusions $j_i:K_i\to K_{i+1}$ with $i=1,\ldots n-1$ such that $\iota=j_{n-1}\circ\ldots\circ j_1$ (see Figure \ref{fig:elemouille}).
\end{itemize}

\begin{figure}[h!]
\centering
\includegraphics[scale=0.1]{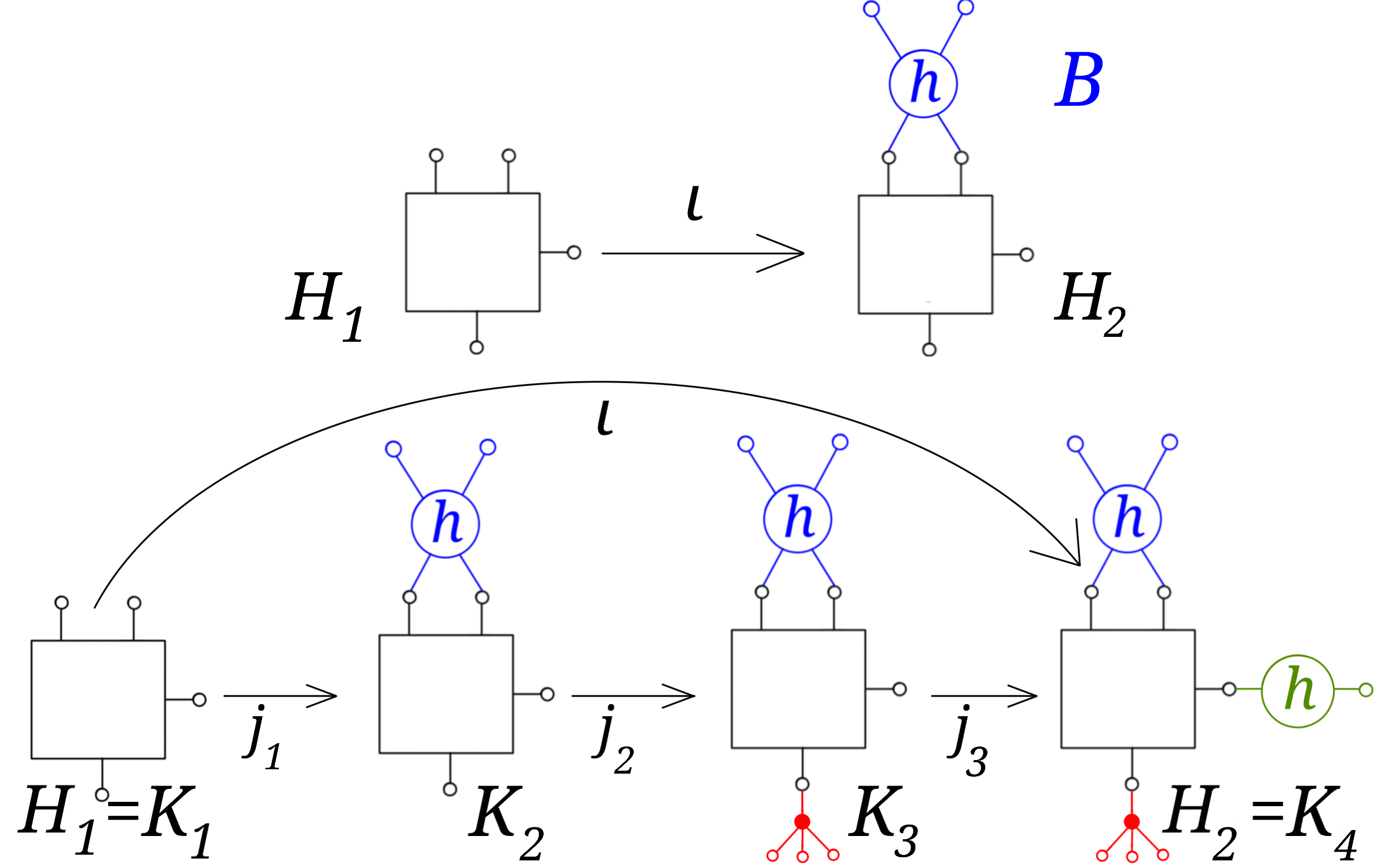}
\caption{An elementary inclusion (upstairs) and an elementarily decomposable inclusion (downstairs): the graph $H_2$ is obtained from $H_1$ by attaching successively an $h_4$-piece, an $s$-piece and an $h_2$-piece.}\label{fig:elemouille}
\end{figure}

\section{Forests of $C$-graphs}\label{ss.forests} As outlined in \S \ref{ss.strategy}, the leaves of the lamination by graphs constructed in Theorem \ref{t.grafos} will be obtained as direct limits of $C$-subgraphs of coverings in the tower. In order to include simultaneously all the desired leaves, we will consider a family of these $C$-subgraphs organized in an ``arborescent structure'' called a forest of $C$-subgraphs. The abstract version of forests of $C$-subgraphs are forests of $C$-graphs. In the first paragraph we introduce these two concepts and the concept of realization of forests in towers which relates them.

\subsection{Forests of $C$-graphs and realization. }\label{ss.forests_inclusion_realization}
Given an oriented graph $\Gamma$ and an edge $e\in E(\Gamma)\subseteq V(\Gamma)^2$, define its \emph{origin} and \emph{terminal} vertices as the vertices $o(e)$ and $t(e)$ so that $e=(o(e),t(e))$.

\paragraph{\textbf{Forests}}

A \emph{forest} is an oriented graph $\mathcal{T}=(V(\mathcal{T}),E(\mathcal{T}))$ where the set $V(\cT)$ of vertices and the set $E(\cT)\dans V(\cT)^2$ of oriented edges satisfy

 \begin{itemize}
 \item $V(\mathcal{T})$ has a countable partition $V(\mathcal{T})=\bigsqcup_{n\in\N}V_n(\mathcal{T})$ were each $V_n(\mathcal{T})$ is finite. We call $V_n(\mathcal{T})$ the $n$-th floor of $\mathcal{T}$. 
 \item $E(\mathcal{T})$ is contained in $\bigcup_{n\in\N} (V_n(\mathcal{T})\times V_{n+1}(\mathcal{T}))$. In other words, given any edge, its terminal vertex is one floor upper than its origin vertex.
\item Every vertex in $\bigcup_{n\geq 1}V_n(T)$ is the terminal vertex of exactly one edge. 
\item Every vertex is the origin vertex of at least one edge.
\end{itemize}

In other words forests are a disjoint union of finitely many trees whose vertices have an integer graduation. We write $E(\mathcal{T})=\bigsqcup_{n\in\N}E_n(\mathcal{T})$ where $E_n(\mathcal{T})=\{e\in E(\mathcal{T}):o(e)\in V_n(\mathcal{T})\}$. We consider all paths in $\cT$ to be reduced, meaning: if $\rho=e_1\ldots e_k$ is a \emph{path} in $\cT$ then $e_i\neq e_{i+1}$ for $i=1,\ldots,k-1$.

\paragraph{\textbf{Forests of $C$-(sub)graphs }}We define a \emph{forest} of $C$-graphs as a triple $$\mathcal{H}=(\mathcal{T},\{H_v\}_{v\in V(\mathcal{T})},\{\iota_e\}_{e\in E(\mathcal{T})})$$ where $\mathcal{T}$ is a forest, $\{H_v\}_{v\in V(\mathcal{T})}$ is a family of finite $C$-graphs and $\{\iota_e\}_{e\in E(\mathcal{T})}$ is a family of $C$-inclusions $\iota_e:H_{o(e)}\to H_{t(e)}$. 

There is a particular case of forest of $C$-graphs which is of particular interest for our purpose. That is the case where the family of $C$-graphs $(H_v)_{v\in V(\mathcal{T})}$ is a family of $C$-subgraphs included in the coverings of a tower. We proceed to give a formal definition. 

\begin{definition}

A \emph{forest of $C$-subgraphs} is a forest of $C$-graphs $$\mathcal{S}=(\mathcal{T},\{G_v\}_{v\in V(\mathcal{T})},\{j_e\}_{e\in E(\mathcal{T})})$$ so that, there exists a tower $\mathbb{U}=\left\{q_n:\Gamma_{n+1}\to \Gamma_n\right\}$ satisfying: 
\begin{itemize}
\item $\{G_v:v\in V_n(\cT)\}$ consists of a disjoint family of finite $C$-subgraphs of $\Gamma_n$;
\item the family of $C$-inclusions $\{j_e:G_{o(e)}\to G_{t(e)}:e\in E(\mathcal{T})\}$ consists of $(1:1)$ lifts. That is $q_n\circ j_e=\Id$ for every $e\in E_n(\mathcal{T})$.
\end{itemize} 
\end{definition}
In this case we say that $\mathcal{S}$ is \emph{included} in the tower $\mathbb{U}$.

\paragraph{\textbf{Realization in towers }} Let $\mathcal{H}=(\mathcal{T},\{H_v\}_{v\in V(\mathcal{T})},\{\iota_e\}_{e\in E(\mathcal{T})})$ be  a forest of $C$-graphs. We say that it is \emph{realized in a tower} if there exist \begin{itemize}
\item a $C$-subgraph forest  $\mathcal{S}=(\mathcal{T},\{G_v\}_{v\in V(\mathcal{T})},\{j_e\}_{e\in E(\mathcal{T})})$ and 
\item a family of collapses $\{f_v:G_v\to H_v:v\in V(\cT)\}$
\end{itemize} 
satisfying $f_{t(e)}\circ j_e=\iota_e\circ f_{o(e)}$ for every $e\in E(\cT)$.

\subsection{Limits of forests of $C$-graphs}\label{ss.limits_forest} In this paragraph we define limits of forest of $C$-(sub)graphs. These are families of $C$-graphs, parametrized by the ends of the underlying forest, which are obtained taking direct limits of sequences of $C$-inclusions. We show (under mild assumptions), that the limits of subgraphs forests embed as leaves in the inverse limit lamination of the underlying tower of coverings. Then, in the case we have a realization of a $C$-graph forest, we show that the ends pairs of the limits of this $C$-graph forest are realized as leaves of ends pairs of the lamination associated to the tower. 

\paragraph{\textbf{Direct limits }} Let $(G_n)_{n\in\N}$ be a sequence of $C$-graphs and $\{\iota_n:G_n\to G_{n+1}\}$ be a sequence of $C$-inclusions. Then, we define the \emph{direct limit} of the sequence as 
$$G_\infty=\underrightarrow{\lim}\{\iota_n:G_n\to G_{n+1}\}= \left.\bigsqcup G_n\right/\sim$$
where $\sim$ is the equivalence relation generated by $\forall x\in G_n,\,x\sim \iota_n(x)$. The space $G_\infty$ is naturally a $C$-graph. Moreover there exists a $C$-inclusion $I_n:G_n\to G_\infty$ such that $I_{n+1}\circ \iota_n=I_n$ for every $n\in\N$. 

We say that $f:G\to H$ is a \emph{collapse onto its image} if \begin{itemize}
\item $f(G)$ is a $C$-subgraph of $H$, and 
\item $f:G\to f(G)$ is a collapse
\end{itemize} 

Direct limits enjoy the following universal property which allow us to construct maps defined on them.

\begin{proposition}[Universal property]\label{p.universal}
Let $G$ be a $C$-graph. Assume that there exists a sequence of maps $\phi_n:G_n\to G$ which satisfy the compatibility condition
$$\phi_n=\phi_{n+1}\circ \iota_n.$$
Then there exists a map $\phi:G_\infty\to G$ such that for every $n\in\N$
$$\phi_n=\phi\circ I_n.$$ Moreover, \begin{enumerate}
\item if the maps $\phi_n$ are $C$-inclusions, so is $\phi$ and
\item if the maps $\phi_n$ are collapses onto its image, so is $\phi$.
\end{enumerate}
\end{proposition}
\begin{proof} Assume first that all $\phi_n$ are $C$-inclusions. Since all the $\phi_n$ are injective so is $\phi$. Also notice that unions of $C$-subgraphs are $C$-subgraphs and therefore $\phi(G_\infty)=\bigcup_{n\in\N}\phi_n(G_n)$ is a $C$-subgraph. Then $\phi$ is a $C$-inclusion as desired. 

Consider now the case where all $\phi_n$ are collapses onto their images. First notice that we can re-define $H$ to be $\phi(G_{\infty})$ and the hypotheses hold. Also, by definition of collapse, for each $n\in\N$ there exist \begin{itemize} 
\item $\mathcal{F}_n$ a disjoint family of finite, connected and non-homologically trivial subgraphs contained in $\text{Int}(G_n)$; 
\item injective maps $\hat\phi_n:G_n/\mathcal{F}_n\to H$ such that $\phi_n=\hat\phi_n\circ \pi_n$ where \\$\pi_n:G_n\to G_n/\mathcal{F}_n$ is the quotient projection.  

\end{itemize}

Since $\hat\phi_n\circ\pi_n=\hat\phi_{n+1}\circ\pi_{n+1}\circ \iota_n$, there exists a map $$\overline{\iota_n}:G_n/\mathcal{F}_n\to G_{n+1}/\mathcal{F}_{n+1}$$ satisfying the equations: \begin{enumerate}
\item\label{ecua1} $\pi_{n+1}\circ \iota_n=\overline{\iota_n}\circ\pi_n$ 
\item\label{ecua2} $\hat\phi_{n}=\hat\phi_{n+1}\circ\overline{\iota_n}$
\end{enumerate} 
From Condition (\ref{ecua1}) we deduce that for each $S\in\mathcal{F}_n$ there exists $S'\in \mathcal{F}_{n+1}$ with $\iota_n(S)\subseteq S'$. Moreover, from Condition (\ref{ecua2}) we deduce that $\overline{\iota_n}$ is injective and therefore $\iota_n(G_n)\cap S'=\iota_n(S)$.  

On the other hand, since $\pi_n(S)\cap\pi_n(\partial G_n)=\emptyset$ and $\overline{\iota_n}$ is injective, Condition (\ref{ecua1}) imply that $\pi_{n+1}(\iota_n(S))\cap \pi_{n+1}(\iota_n(\partial G_n))=\emptyset$. Taking preimages we get that $S'\cap \iota_n(\partial G_n)=\emptyset$. Then, since $\iota_n(\partial G_n)=\partial \iota_n(G_n)$ we can decompose $$S'=[\text{Int}(\iota_n(G_n))\cap S']\cup[\iota_n(G_n)^c\cap S'].$$ On the other hand, since $S'$ is connected we conclude that $\iota_n(S)=S'$. This implies that $I_n(\mathcal{F}_n)\subseteq I_{n+1}(\mathcal{F}_{n+1})$. Therefore $\mathcal{F}_{\infty}:=\cup_{n\in\N}I_n(\mathcal{F}_n)$ is a disjoint family of finite, connected and non-homologically trivial subgraphs contained in $\text{Int}(G_{\infty})$. Then, by the universal property of the quotient, $\phi=\hat\phi\circ\pi_\infty$ where $\pi_\infty:G_\infty\to G_{\infty/\mathcal{F}_\infty}$ is the quotient projection and $\hat\phi$ is an isomorphism. Also, since $\phi_n$ preserve vertex types when restricted to $G_n\setminus \cup_{S\in\mathcal{F}_n}S$ it holds that $\phi$ does the same when restricted to $G_\infty\setminus \cup_{S\in\mathcal{F}_\infty}S$. This finishes the proof of the Proposition. 
\end{proof}

\paragraph{\textbf{Limits of forests of $C$-graphs} } Now we are ready to define the \emph{limits} of a forest of $C$-graphs. For this consider a forest of $C$-graphs $$\mathcal{H}=(\cT,\{G_v\}_{v\in V(\cT)},\{j_e\}_{e\in E(\cT)})$$ and an end $\alpha\in\cE(\cT)$. Denote $\rho^\alpha$ the semi-infinite ray starting at $V_0(\cT)$ and converging to $\alpha$. Write $\rho^\alpha=(e_i)_{i\in\N}$. We define the \emph{limit of $\mathcal{H}$ associated to} $\alpha$ as the $C$-graph 
\begin{equation}\label{eq.gend}
G^{\alpha}=\underrightarrow{\lim}\{j_{e_n}:G_{o(e_n)}\to G_{t(e_n)}\}.
\end{equation}
Finally, we define the \emph{limits of} $\mathcal{H}$ as the family $\{G^{\alpha}:\alpha\in\cE(\cT)\}$. 

\paragraph{\textbf{Limits of forests of $C$-subgraphs and leaves}}Consider a forest of $C$-subgraphs $$\mathcal{S}=(\mathcal{T},\{G_v\}_{v\in V(\mathcal{T})},\{j_e\}_{e\in E(\mathcal{T})})$$ included in a tower $\mathbb{U}=\{q_n:\Gamma_{n+1}\to\Gamma_n\}$ and let $\mathcal{M}$ denote the inverse limit of $\mathbb{U}$. We will show how to embed the limits of $\mathcal{S}$ in the leaves of $\mathcal{M}$.

\begin{proposition}\label{p.iso} 
Assume that $j_e(G_{o(e)})\subseteq \Int(G_{t(e)})$ for every edge $e\in E(\cT)$. Then, for every end $\alpha\in\cE(\cT)$, there exists a leaf $\cL^\alpha$ of $\mathcal{M}$ and an isomorphism of $C$-graphs $$\phi^\alpha:G^\alpha\to\cL^\alpha,$$
where $G_\alpha$ is the limit graph defined by \eqref{eq.gend}.
\end{proposition}

\begin{proof}
For this, consider $\alpha\in\cE(\cT)$ defined by a semi-infinite ray $\rho^\alpha=(e_i)_{i\in\N}$ starting at $V_0(\cT)$. Rewrite $G_n:=G_{o(e_n)}$,  $j_n:=j_{e_n}$ and notice that $G^{\alpha}=\underrightarrow{\lim}\{j_{e_n}:G_{n}\to G_{n+1}\}$. We proceed to define a family of $C$-inclusions $\phi^\alpha_n:G_n\to\mathcal{M}$. For this, if $n_1<n_2$ denote \begin{itemize}
\item $J_{n_1n_2}=j_{n_2-1}\circ\ldots\circ j_{n_{1}}:G_{n_1}\to G_{n_2}$ and 
\item $Q_{n_1n_2}=q_{n_1}\circ\ldots\circ q_{n_2-1}:\Gamma_{n_2}\to\Gamma_{n_1}$. 
\end{itemize}
Then, define $\phi^\alpha_n:G_{n}\to\mathcal{M}$ as $\phi^\alpha_n(x)=(x_k)_{k\in\N}$ where
\begin{itemize}
\item $x_k=x$ if $k=n$
\item $x_k=J_{nk}(x)$ if $k>n$
\item $x_k=Q_{kn}$ if $k<n$
\end{itemize} 
It follows directly from its definition that $\{\phi^\alpha_n:G_n\to\mathcal{M}\}_{n\in\N}$ is a family of continuous maps satisfying $\phi^\alpha_{n+1}\circ j_n=\phi^\alpha_n$. Then, by the universal property we get a continuous map $\phi:G^\alpha\to\cM$. Moreover, since the image of $\phi$ is connected, it is contained in a leaf of $\mathcal{M}$ that we denote $\cL^\alpha$. Notice that the maps $\phi_n:G_n\to\cL^\alpha$ are $C$-inclusions. Therefore, we can apply the universal property for $C$-inclusions to obtain a $C$-inclusion $\phi^\alpha:G^\alpha\to\cL^\alpha$. 

Suppose now that $j_e(G_{o(e_n)})\subseteq \text{Int}(G_{t(e_n)})$ for every $e\in E(\cT)$. This implies that $\text{Im}(\phi^\alpha_n)\subseteq \text{Int}(\phi^\alpha_{n+1})$ for every $n\in\N$ and therefore $\text{Im}(\phi^\alpha)$, which equals $\bigcup_{n\in\N}\text{Im}(\phi^\alpha_n)$, is both open and close. Then, in this case we have that $\phi^\alpha$ is an isomorphism of $C$-graphs.
\end{proof}

\subsection{Realization and leaves}
In order to prove our next Proposition we need the following Lemma.
\begin{lemma}\label{l.interior} Consider $G_1,G_2,H_1$ and $H_2$ $C$-subgraphs. Also consider \begin{itemize}
\item $C$-inclusions $j:G_1\to G_2$ and $\iota:H_1\to H_2$; and
\item collapses $h_i:G_1\to H_i$ for $i=1,2$
\end{itemize}
such that $h_2\circ j=\iota\circ h_1$ and $\iota(H_1)\subseteq\Int(H_2)$.

 Then $j(G_1)\subseteq\Int(G_2)$.
\end{lemma}
\begin{proof} First notice that, since $h_2(j(G_1))\subseteq \iota(H_2)$ we obtain that $j(G_1)\subseteq h_2^{-1}(\iota(H_1))$. Also, since $h_2$ is a collapse we get that $h_2^{-1}(\text{Int}(H_2))\subseteq\text{Int}(G_2)$. Finally, by hypothesis we get that $h_2^{-1}(\iota(H_1))\subseteq h_2^{-1}(\text{Int}(H_2))$.

Putting all this together we conclude that $j(G_1)\subseteq \text{Int}(G_2)$ as desired.
\end{proof}

\begin{proposition} \label{p.include}Consider a forest $\mathcal{H}=(\mathcal{T},\{H_v\}_{v\in V(\mathcal{T})},\{\iota_e\}_{e\in E(\mathcal{T})})$ satisfying $\iota_e(H_{o(e)})\subseteq \Int(H_{t(e)})$ for every $e\in E(\cT)$. Assume that $\mathcal{H}$ is realized in a tower $\mathbb{U}$ with associated lamination $\mathcal{M}$. Then, for each $\alpha\in\cE(\cT)$ there exists a leaf of $\mathcal{M}$ denoted $\cL^{\alpha}$ such that $$(\cE_0(H^\alpha),\cE(H^\alpha))=(\cE_0(\cL^\alpha),\cE(\cL^\alpha)).$$
\end{proposition}
\begin{proof}Write $\mathcal{U}=\{q_n:\Gamma_{n+1}\to\Gamma_n:n\in\N\}$. Since $\mathcal{H}$ is realized in $\mathbb{U}$ there exist:
\begin{itemize}
\item a forest of $C$-subgraphs $\mathcal{S}=(\mathcal{T},\{G_v\}_{v\in V(\mathcal{T})},\{j_e\}_{e\in E(\mathcal{T})})$ where $G_v$ is a $C$-subgraph of $\Gamma_n$ for every $v\in V_n(\cT)$ and $n\in\N$;
\item a family of collapses $\{h_v:G_v\to H_v:v\in V(\mathcal{T})\}$ satisfying $h_{t(e)}\circ j_e=\iota_e\circ h_{o(e)}$ for every $e\in E(\cT)$.
\end{itemize}

We first construct a family of collapses $\{h^\alpha:G^\alpha\to H^\alpha:\alpha\in\cE(\cT)\}$. For this, take $\alpha\in\cE(\cT)$ and $\rho=(e_i)_{i\in\N}$ the semi-infinite ray in $\cT$ converging $\alpha$ with $o(e_0)\in V_0(\cT)$. Recall that \begin{itemize}
\item $G^{\alpha}=\underrightarrow{\lim}\{j_{e_n}:G_{o(e_n)}\to G_{o(e_{n+1})}\}$,

\item $H^{\alpha}=\underrightarrow{\lim}\{\iota_{e_n}:H_{o(e_n)}\to H_{o(e_{n+1})}\}$ and
\end{itemize}
denote $I_n:H_{o(e_n)}\to H^\alpha$ and $J_n:G_{o(e_n)}\to G^\alpha$ the corresponding $C$-inclusions. Define $\psi_n^\alpha:G_n\to H^\alpha$ as $\psi^\alpha_n=I_n\circ h_{o(e_n)}$. Since $h_{o(e_n)}$ is a collapse and $I_n$ is a $C$-inclusion we have that $\psi^\alpha_n$ is a collapse onto its image. On the other hand, since $\iota_{e_n}\circ h_{o(e_n)}=h_{t(e_n)}\circ j_{e_n}$ and $I_n=I_{n+1}\circ \iota_{e_n}$ we conclude that $\psi_n^\alpha=\psi_{n+1}^\alpha\circ j_{e_n}$. Therefore, by Proposition \ref{p.universal}, there exists a collapse onto its image $\psi^\alpha:G^\alpha\to H^\alpha$. To show that $\psi^\alpha$ is indeed an isomorphism note that $$\text{Im}(\psi^\alpha)=\bigcup_{n\in\N}\text{Im}(\psi^\alpha_n)=\bigcup_{n\in\N}\text{Im}(I_n)=H^\alpha.$$ Then, applying Proposition \ref{p.codingends} we deduce that the end pairs $(\cE_0(H^\alpha),\cE(H^\alpha))$ and $(\cE_0(G^\alpha),\cE(G^\alpha))$ are equivalent. 

On the other hand, since $\iota_e(H_{o(e)})\subseteq\text{Int}(H_{t(e)})$ for every $e\in E(\cT)$, Lemma \ref{l.interior} implies that $j_e(G_{o(e)})\subseteq\text{Int}(G_{t(e)})$ for every $e\in E(\cT)$. Then, we are in condition to apply Proposition \ref{p.iso} to the forest $\mathcal{S}$ and find, for each $\alpha\in\cE(\cT)$ a leaf of $\mathcal{M}$ denoted by $\cL^\alpha$  which is isomorphic to $G^\alpha$. In particular, it holds that the end pairs $(\cE_0(G^\alpha),\cE(G^\alpha))$ and $(\cE_0(\cL^\alpha),\cE(\cL^\alpha))$ are equivalent. This finishes the proof of the Proposition. 

\end{proof}

\section{Surgeries and the main Lemma}\label{s.surgeries}
The main result of this section is Lemma \ref{l.main} where we show how to realize some families of $C$-graph forests in towers. The main ingredients in the proof of this lemma are Lemmas \ref{l.bblock} and \ref{l.multiplelift} which allow us to perform the inductive step. These lemmas heavily rely on the surgery operation, which we proceed to define.  

\subsection{Surgeries of finite covers}\label{ss.surgery}

Consider a $C$-graph $\Gamma$ together with a subset $X\subseteq\Gamma$ consisting of $b_2$-vertices. We define $\Gamma_X$ as the $C$-graph obtained by \emph{cutting $\Gamma$ along} the vertices in $X$. Namely, there exists a map $j_X:\Gamma_X\to\Gamma$ such that $j_X|_{\Gamma_X\setminus X}$ is $(1:1)$ onto $\Gamma\setminus X$ and each $a\in X$ has two preimages. In other words, each vertex in $X$ splits into two $b_1$-vertices.   

We proceed to define the operation of surgery. For this, consider a finite covering $p_0:\Gamma\to\Gamma_0$ together with finite subsets of $b_2$-vertices $X\subseteq\Gamma$ and $X_0\subseteq\Gamma_0$ such that $p_0|_{X}$ is $(2:1)$ onto $X_0$. For each $a\in X_0$ denote $e^-_{a}$ and $e^+_{a}$ the edges of $\Gamma_0$ adjacent to $a$. 

Now, for each $a\in X_0$ denote: 
\begin{itemize}
\item $a_1$, $a_2$ its preimages in $X$ under $p_0$;
\item $e_{a_i}^+$, $e_{a_i}^-$ the preimages of $e^+_a$ and $e^-_a$ which are adjacent to $a_i$; 
\item $a_i^{\pm}$ the copies of $a_i$ in $\Gamma_{X}$ which are adjacent to $j_X^{-1}(e_{a_i}^{\pm})$. 
\end{itemize}

Define $X^+=\{a_i^+:a\in X_0,i=1,2\}$ and $X^-=\{a_i^-:a\in X_0,i=1,2\}$. Then, we define $\Gamma^X$ as the quotient of $\Gamma_X$ given by the equivalence relation $a_1^+\sim a_2^-$, $a_1^-\sim a_2^+$ for every $a\in X_0$. Denote $j^X:\Gamma_X\to\Gamma^X$ the quotient projection. Since $p_0\circ j_X$ passes to the quotient, we can define $p_X:\Gamma^X\to\Gamma_0$ satisfying $p_X\circ j^X=p_0\circ j_X$. It is straightforward to check that $p_X$ is indeed a covering map. We call $p_X$ the \emph{surgery of $p_0$ along X}. We refer to Figure \ref{fig:surjibouille}.  Note that non-connected covering spaces can become connected after surgery.

\begin{figure}[h!]
\centering
\includegraphics[scale=0.12]{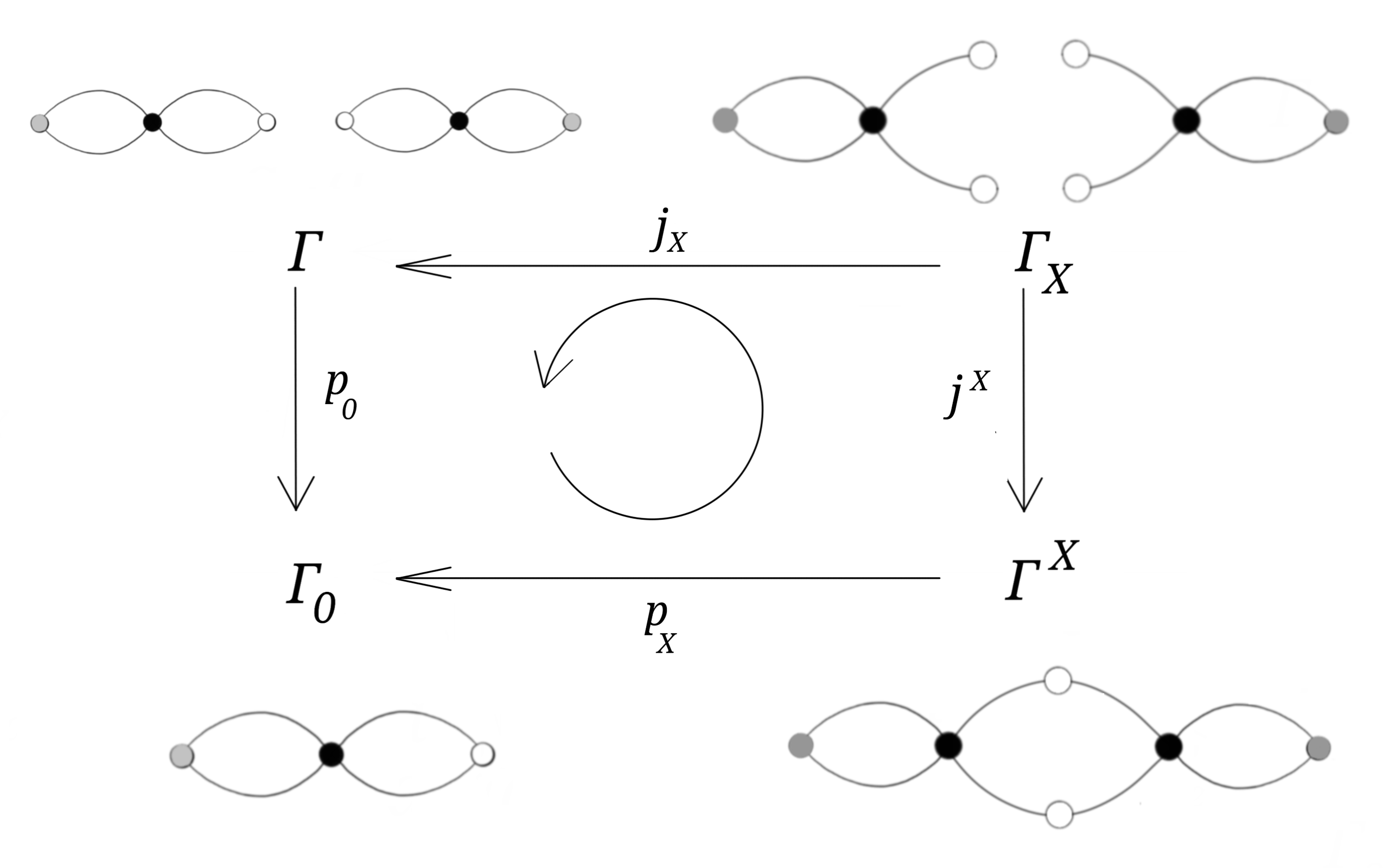}
\caption{The surgery of a \emph{non-connected} $2$-fold covering $\Gamma$ of the figure eight $C$-graph $\Gamma_0$ along a set $X$ of $2$ boundary vertices pictured in white.}\label{fig:surjibouille}
\end{figure}

\begin{remark}\label{r.connected} We point out that boundary vertices do not disconnect finite coverings of the figure eight $C$-graph. To see this, consider a finite covering $p:\Gamma\to\Gamma_0$ and a boundary vertex $v\in\Gamma$. Let $\gamma_0$ be a simple closed curve through $p(v)$. Since $p$ is finite, the connected component of $p^{-1}(\gamma_0)$ through $v$ is also a simple closed curve that we denote $\gamma$. On the other hand, since $v$ is a boundary vertex, it holds that $v$ has a neighbourhood $U$ such that $U\setminus \{v\}$ has two connected components. Finally, since $\gamma\setminus \{v\}$ joins this two components, the remark follows. 
\end{remark}

\subsection{Two important lemmas}

The next two lemmas will be the two building blocks in order to construct towers of coverings with an a priori fixed forest of $C$-graphs included in it. The first lemma shows how to ``realize'' an elementary $C$-inclusion by a covering map: we call it an \emph{elementary realization}. The second lemma shows how to construct coverings to \emph{replicate} subgraphs, which is necessary to realize simultaneously various coverings.

\begin{lemma}[Elementary realization]\label{l.bblock} Consider $\Gamma$ a finite and connected covering space of the figure eight $C$-graph satisfying $\beta_1(\Gamma)\geq 3$. Assume that $G,G_0$ are disjoint subgraphs of $\Gamma$, $f:G\to H$ is a collapse and $\iota: H\to H'$ is an elementary $C$-inclusion. Then, there exist 
\begin{itemize}
\item a connected and finite covering $p:\hat\Gamma\to\Gamma$ and
\item disjoint subgraphs $\hat G,\hat G_0\subseteq\hat\Gamma$ satisfying: 
\begin{itemize}
\item $\hat G_0$ is a $(1:1)$-lift of $G_0$
\item there exists a $C$-inclusion $j:G\to\hat G$ and a collapse $\hat f:\hat G\to H'$ such that $p\circ j=\Id_G$ and $\iota\circ f=\hat f\circ j$. 
\end{itemize}
\end{itemize}
\end{lemma}
\begin{proof} Write $$H'=\iota(H)\cup B$$ According to the definition of elementary $C$-inclusion we divide the proof in three cases according to the type of $B$.\\

\noindent\emph{Case 1. $B$ is an $h_2$-piece meeting $\iota(H)$ in exactly $1$ boundary vertex .} \\Consider the vertex $v\in H$ such that $\iota(v)=\iota(H)\cap B$. By Remark \ref{r.proper} we have that $f^{-1}(v)$ is a single $b_1$-vertex that we denote by $a$.  

Define $\Delta=\Gamma^{(1)}\sqcup \Gamma^{(2)}$ as the disjoint union of two copies of $\Gamma$ and $q:\Delta\to\Gamma$ the natural covering. Denote $X=\{a_1,a_2\}\subseteq\Delta$ the copies of $a$ in $\Gamma^{(1)}$ and $\Gamma^{(2)}$ respectively. Note that $\Delta_X$ is naturally homeomorphic to $\Gamma^{(1)}_{a_1}\sqcup\Gamma^{(2)}_{a_2}$. 

Consider $p:\Delta^X\to\Gamma$ the surgery of $q$ along $X$, and $$j_X:\Delta_X\to\Delta, \mbox{ } j^X:\Delta_X\to\Delta^X$$ the maps given in the definition of surgery. Notice that $j^X|_{\Gamma^{(i)}_{a_i}}$ is injective for $i=1,2$. 

\begin{figure}[h!]
\centering
\includegraphics[scale=0.14]{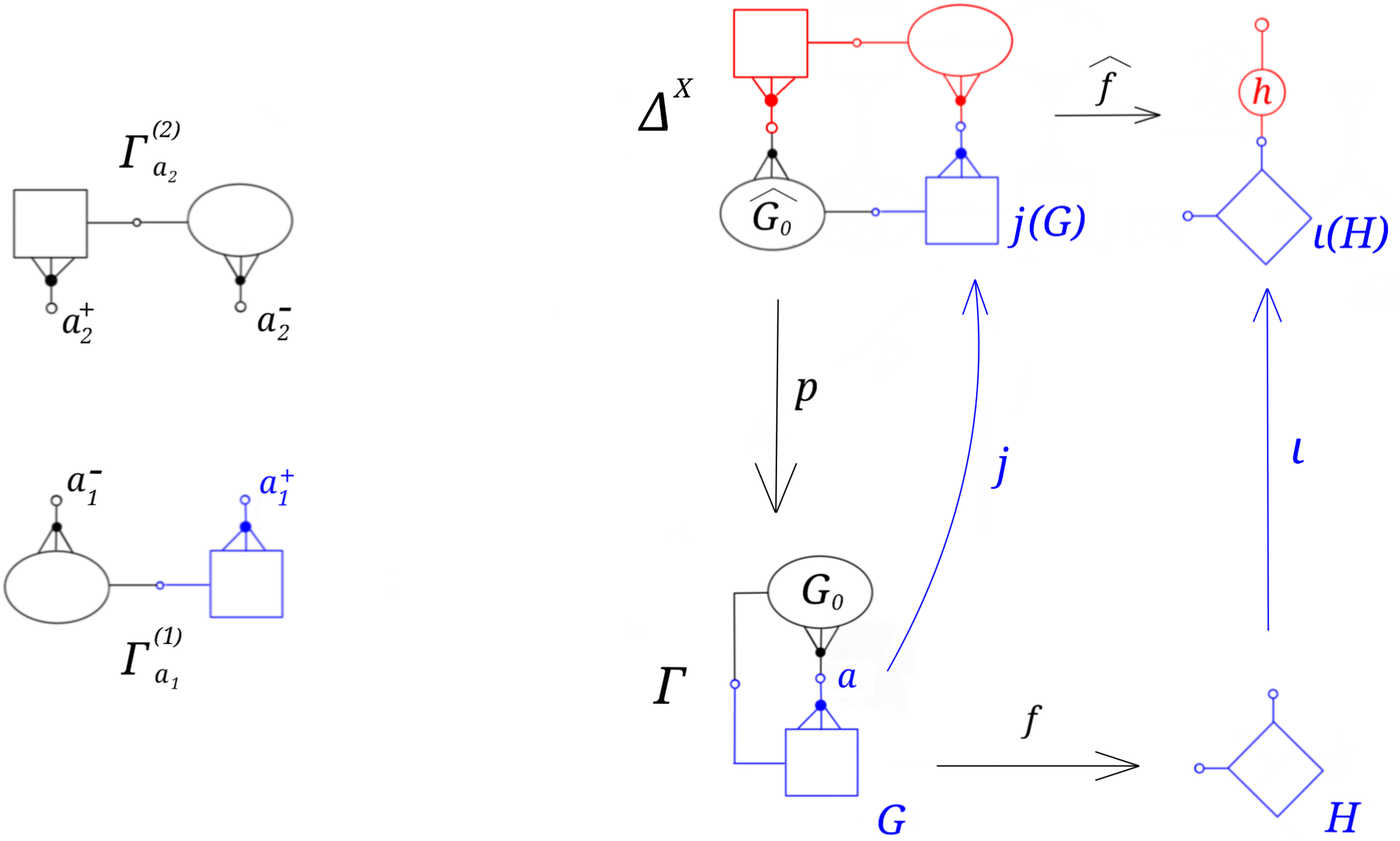}
\caption{Elementary realization in Case 1. The graph $G$, represented by some blue ``square'' graph is collapsed onto a $C$-graph $H$, represented by the blue rhombus. We realize the operation of attaching an $h_2$-piece to $H$ by performing a double covering of $\Gamma$ and attaching to the blue lift of $G$ the red graph, which has nontrivial homology, at a boundary vertex. This picture also serves for Case 2. In that case, instead of attaching to $j(G)$ the whole red graph, we only attach the red $s$-piece.}\label{fig:case1}
\end{figure}

Since $\Gamma\setminus \{a\}$ is connected (see Remark \ref{r.connected}) so is $\Delta^{X}$ which consists of the glueing of two copies of $\Gamma\setminus\{a\}$. Define $j=j^X\circ j_0$ where $j_0$ is the natural embedding of $G$ in $\Gamma^{(1)}_{a_1}$ (recall that $a$ is a $b_1$-vertex of $G$) and note that $j$ is a $(1:1)$-lift of $G$ under $p$. The same argument shows the existence of a $(1:1)$-lift of $G_0$ under $p$ that we denote by $\hat G_0$.

On the other hand, $G':=j^X(\Gamma^{(2)}_{a_2})$ is a subgraph of $\Delta^X$ with non-trivial homology and exactly two  $b_1$-vertices which meets $j(G)$ at $j(a)$. Therefore, we can define the subgraph $\hat{G}=j(G)\cup G'$ and a collapse $\hat f:\hat G\to H'$ sending $G'$ to $B$ and satisfying $\iota\circ f=\hat f\circ j$. This finishes the proof of the lemma in this case.

\noindent\emph{Case 2. $B$ is an $s$-piece meeting $\iota(H)$ in exactly $1$ boundary vertex.}\\  
In this case we use the previous construction and notation but we change the definition of $G'$. For this, let $c$ denote the $s$-vertex adjacent to $j(a)$ in $j^X(\Gamma^{(2)}_{a_2})$. Then, define $G'$ as the ball of radius $1$ around $c$. Note that $G'$ is an $s$-piece contained in $j^X(\Gamma^{(2)}_{a_2})$. Moreover, since $j^X(\Gamma^{(2)}_{a_2})$ and $j(G)$ meet at $j(a)$, so do $j(G)$ and $G'$. Therefore we can define $\hat{G}:=j(G)\cup G'$, and the collapse $\hat f:\hat G\to H'$ sending $G'$ to $B$ and satisfying $\iota\circ f=\hat f\circ j$.\\

\noindent\emph{Case 3. $B$ is an $h_4$-piece meeting $\iota(H)$ in exactly $2$ boundary vertices.}\\ Let $v_1,v_2$ denote the vertices in $B\cap \iota(H)$ and $a_i=f^{-1}(\iota^{-1}(v_i))$. Since $\beta_1(\Gamma)\geq 3$, there must exist a connected component of $\Gamma_{a_1,a_2}$ with non-trivial homology, denote $\Gamma_{a_1,a_2}^0$ such component. On the other hand, since $\Gamma_{a_1,a_2}^0$ has non-trivial homology, there must exist a $b_2$-vertex $a_0\in\Gamma_{a_1,a_2}^0$ which is non-disconnecting. This implies that (the copy of) $a_0$ does not disconnect $\Gamma_{a_1}$ nor $\Gamma_{a_2}$. In other words $\Gamma_{a_0,a_1}$ and $\Gamma_{a_0,a_2}$ are connected. Re-define $a_0$ as the copy of $a_0$ in $\Gamma$. 

Now, define $\Delta=\Gamma^{(1)}\sqcup\Gamma^{(2)}\sqcup\Gamma^{(3)}$ as three disjoint copies of $\Gamma$ and $q:\Delta\to\Gamma$ the associated covering. Define $$X=\{a_0^1,a_1^1,a_0^2,a_2^2,a_1^3,a_2^3\}$$ where $a_i^j$ is the copy of $a_i$ in $\Gamma^{(j)}$. Since each $a_i\in \{a_0,a_1,a_2\}$ has exactly two preimages in $X$ we can define $p:\Delta^X\to\Gamma$ the surgery of $q$ along $X$ and $j^X:\Delta_X\to\Delta^X$. 

\begin{figure}[h!]
\centering
\includegraphics[scale=0.18]{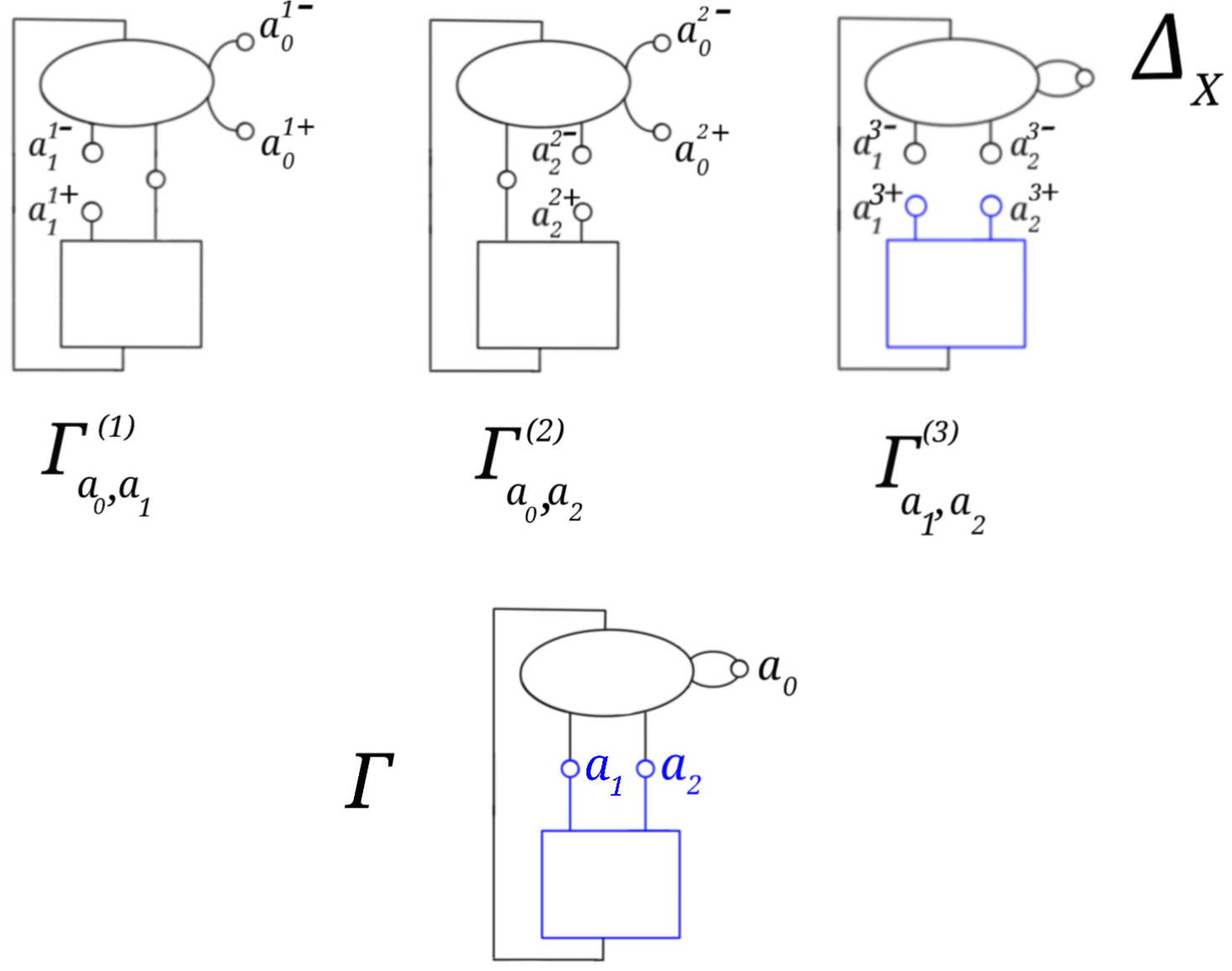}
\caption{The graph $\Delta_X$.}\label{fig:deltitouille}
\end{figure}

We have that in this case, $\Delta_X$ decomposes as 

$$\Delta_X=\Gamma^{(1)}_{a^1_0,a^1_1}\sqcup\Gamma^{(2)}_{a^2_0,a^2_2}\sqcup\Gamma^{(3)}_{a^3_1,a^3_2}$$
(see Figure \ref{fig:deltitouille}). The graphs $\Gamma^{(1)}_{a^1_0,a^1_1}$ and $\Gamma^{(2)}_{a^2_0,a^2_2}$ are connected and $j^X((a_0^1)^{\pm})=j^X((a_0^2)^{\mp})$ by definition of surgery, we conclude that 
$$G':=j^X(\Gamma^{(1)}_{a^1_0,a^1_1}\cup\Gamma^{(2)}_{a^2_0,a^2_2})$$
is connected and contains $4$ $b_1$-vertices which consists on two copies of $a_1$ and two copies of $a_2$ (in Figure \ref{fig:deltitouille}, this graph is obtained after glueing the first two graphs along the vertices $(a_0^{i})^\pm$). Arguing as in the first case, we can consider $j_0:G\to\Gamma^{(3)}_{a^3_1,a^3_2}$ the natural embedding of $G$ in $\Gamma^{(3)}_{a^3_1,a^3_2}$ (its image is the blue subgraph of $\Delta_X$ in Figure \ref{fig:deltitouille} ) and $j=j^X\circ j_0$. Define $\hat{G}=j(G)\cup G'$ and notice that $j(G)$ and $G'$ meet exactly at $j(a_1)$ and $j(a_2)$. Since $B$ is an $h_4$-piece meeting $\iota(H)$ at $\{v_1,v_2\}$, we can define a collapse $\hat f:\hat G\to H'$ sending $G'$ to $B$ and satisfying $\iota\circ f=\hat f\circ j$. Finally, notice there exists a $(1:1)$ lift $\hat G_0$ of $G_0$ in $j^X(\Gamma^{(3)}_{a^3_1,a^3_2})$ (we refer to Figure \ref{fig:starwars}). This finishes the proof of the Lemma.

\begin{figure}[h!]
\centering
\includegraphics[scale=0.12]{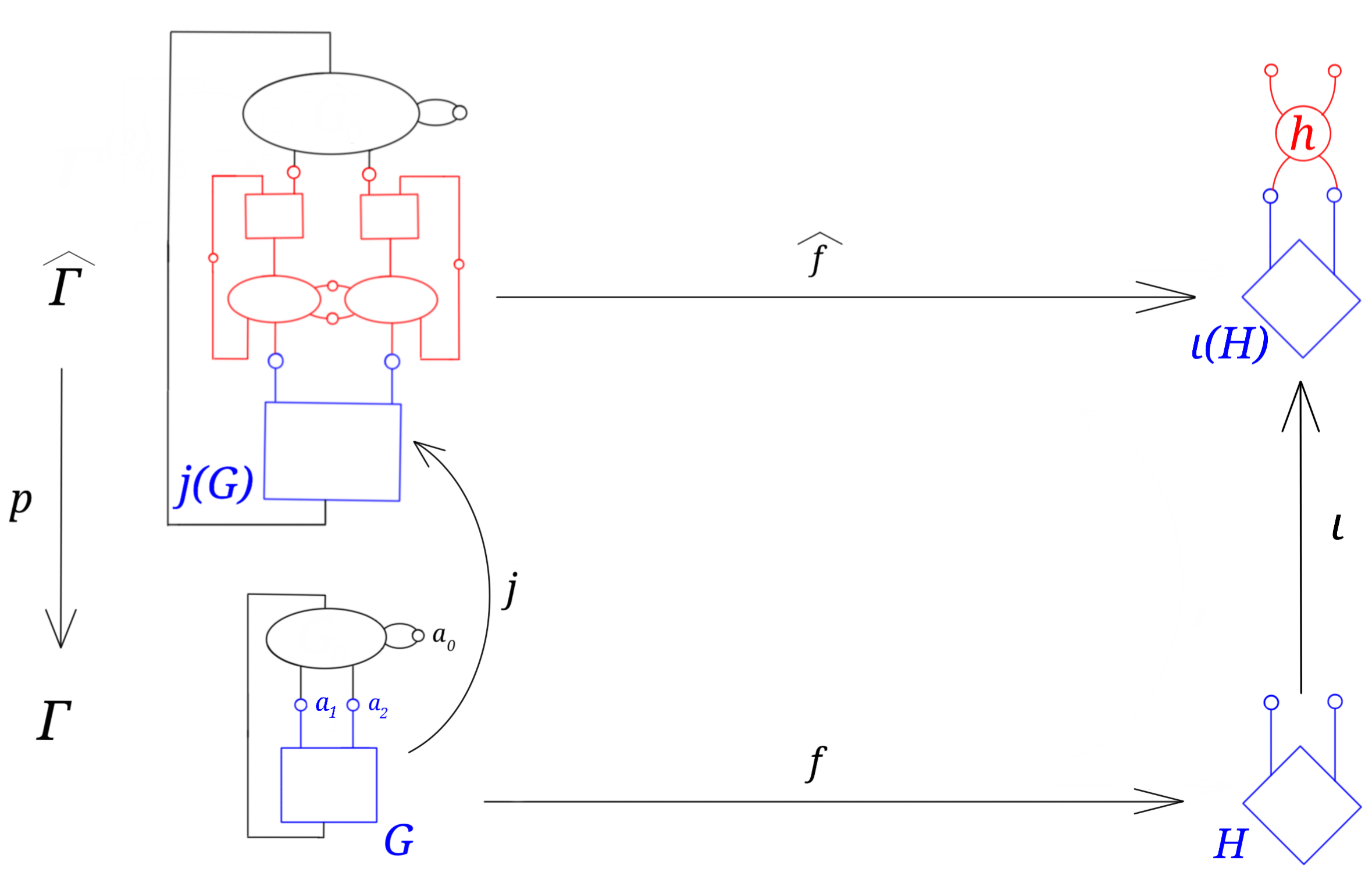}
\caption{Elementary realization in Case 3. We realize the operation of attaching an $h_4$-piece to $H$ by performing a $3$-fold covering of $\Gamma$ and attaching to the blue lift of $G$ the red graph, which has nontrivial homology, at a pair of boundary vertices.}\label{fig:starwars}
\end{figure}

\end{proof}

\begin{lemma}[Replicate]\label{l.multiplelift} Consider $\Gamma$ a finite and connected covering space of the figure eight $C$-graph. Assume that $G_1,\ldots,G_n$ is a family of pairwise disjoint $C$-subgraphs of $\Gamma$ and let $m_1,\ldots,m_n$ be integers satisfying $m_k\geq 1$ for $k=1,\ldots,n$. 

Then, there exists a finite covering $p:\hat\Gamma\to \Gamma$ with a family of different $(1:1)$ lifts $$j_k^l:G_k\to \hat\Gamma$$ where $1\leq k\leq n$ and $1\leq l\leq m_k$. 
\end{lemma}
\begin{proof} Since $\Gamma$ is connected and the subgraphs $G_i$ are disjoint, there exists a vertex $a$ not belonging to any of the $G_i$. We are going to define a new covering with a slight variation of the surgery operation defined in \S\ref{ss.surgery}. Set $N=m_1+\ldots+m_n$. Then, define $$\Delta=\bigsqcup_{i=1}^{N}\Gamma^{(i)}$$ and consider $q:\Delta\to\Gamma$ the associated covering. Denote $X=\{a_1,\ldots,a_N\}$ where $a_{i}$ is the copy of $a$ in $\Gamma^{(i)}$. In this case $\Delta_X=\bigsqcup_{i=1}^{N}\Gamma^{(i)}_{a_{i}}$ where each $\Gamma^{(i)}_{a_{i}}$ contains two copies of $a_i$ that we denote $a_i^{+}$ and $a_i^{-}$. In this case we define $\Delta^X$ as the quotient of $\Delta_X$ under the equivalence relation generated by $$a^+_i\sim a^-_{i+1}\mod N\mbox{ with }i=1,\ldots,N.$$ 

Denote $j^X:\Delta_X\to\Delta^X$ the quotient map. Note that $q\circ j_X$ factors through $j^X$ defining a finite and connected covering $p:\Delta^X\to\Gamma$ which consists of the ``cyclic'' glueing of $N$ copies of $\Gamma_a$ (see Figure \ref{fig:replicouille}). Then, label these copies as $$\Delta^X=\bigcup_{1\leq k\leq n; 1\leq l\leq m_k}C_{k,l}$$ Since $a\notin\bigcup_{i=1}^{N} G_i$ there exists a $(1:1)$ lift of $j^l_k:G_k\to C_{k,l}$ for every $1\leq k\leq n; 1\leq l\leq m_i$. This finishes the proof of the Lemma. \end{proof}

\begin{figure}[h!]
\centering
\includegraphics[scale=0.1]{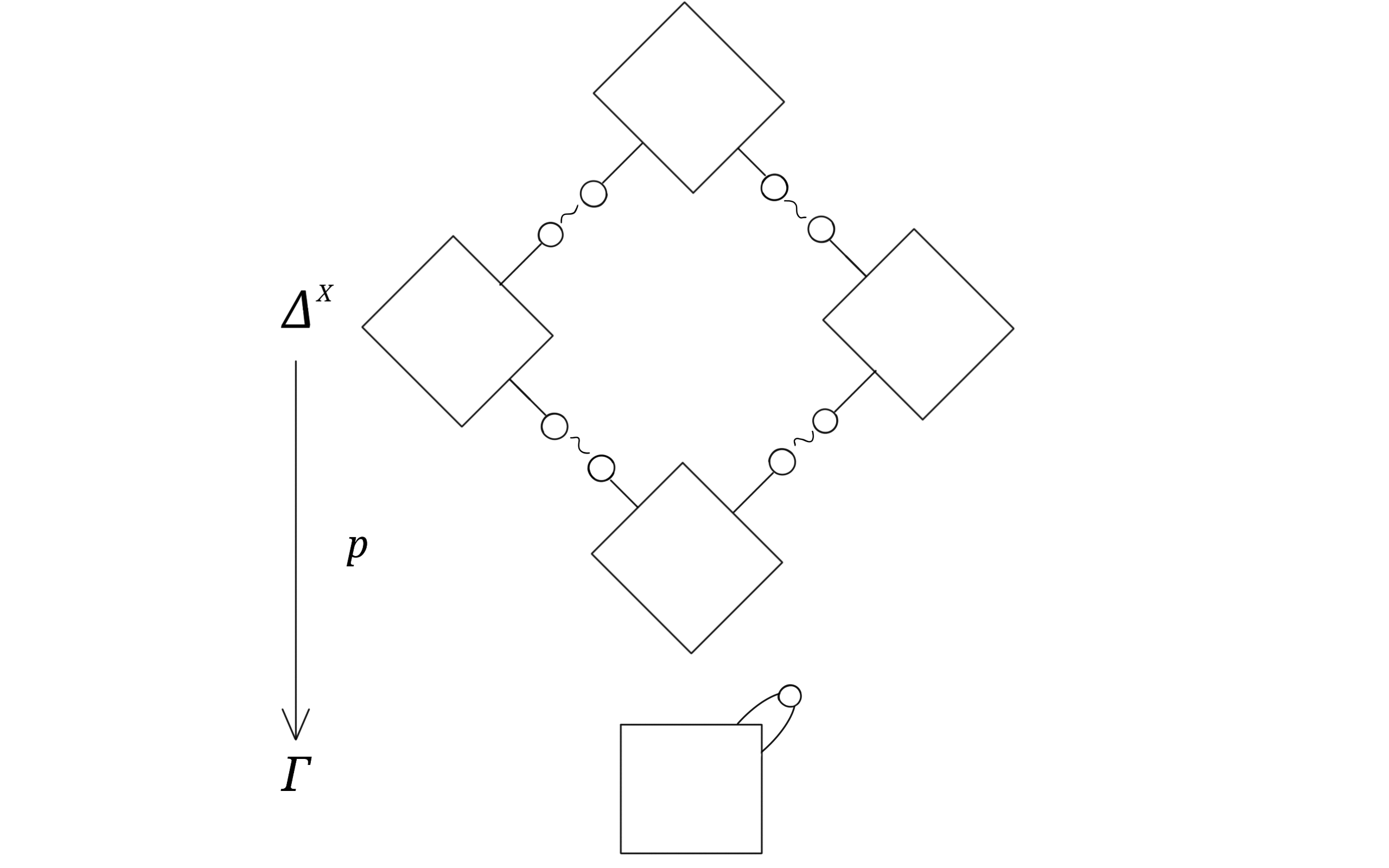}
\caption{Here the vertex $a$ is pictured in white. The graph $\Delta^X$ is obtained by cyclic surgery of various copies of $\Gamma$ along copies of $a$.}\label{fig:replicouille}
\end{figure}

\subsection{The main Lemma} When the $C$-inclusions in a forest of $C$-graphs are elementarily decomposable, we say that we have an \emph{elementarily decomposable forest}. The main Lemma says that (under some assumptions), elementarily decomposable forests can be realized in towers. To prove this Lemma we first need to \emph{decompose} the forest. After the decomposition we will use the previous lemmas on surgeries to realize our decomposed forest through an inductive process.
\paragraph{\textbf{Compositions and decompositions }}Consider a strictly increasing map $\sigma:\N\to\N$ and a forest $\cT=(V(\cT),E(\cT))$. Then, we define the $\sigma$-\emph{composition} of $\cT$ as the forest $\cT_\sigma=(V(\cT_\sigma),E(\cT_\sigma))$ where 

\begin{itemize} 
\item $V(\cT_{\sigma})=\sqcup_{n\in\N}V_{\sigma(n)}\cT$ and 
\item $E_n(\cT_\sigma)=\{\rho:\rho \text{ is a path in }\cT\text{ joining }V_{\sigma(n)}\cT\text{ and }V_{\sigma(n+1)}\cT\}$
\end{itemize}

Also, given a forest of $C$-graphs $\mathcal{H}=(\mathcal{T},\{H_v\}_{v\in V(\mathcal{T})},\{\iota_e\}_{e\in E(\mathcal{T})})$ we define the $\sigma$-composition of $\mathcal{H}$ as the forest 
$$\mathcal{H}_\sigma=(\mathcal{T}_\sigma,\{H_v\}_{v\in V(\mathcal{T}_\sigma)},\{\iota_\rho\}_{\rho\in E(\mathcal{T}_\sigma)})$$
with $\iota_\rho:=\iota_{e_k}\circ\ldots\circ \iota_{e_1}$ where $\rho=e_1\ldots e_k$ belongs to $E(\cT_\sigma)$. 

If for some increasing map $\sigma$ we have that $\mathcal{H}_1$ is a $\sigma$-composition of $\mathcal{H}_2$ we say that $\mathcal{H}_2$ is a \emph{decomposition} of $\mathcal{H}_1$. 

The proof of the following proposition follows directly from the definitions. 
\begin{proposition}\label{p.decomposition} Consider an elementarily decomposable forest of $C$-graphs\\ $\mathcal{H}=(\cT,\{H_v\}_{v\in V(\cT)},\{\iota_e\}_{e\in E(\cT)})$. Then, there exists a decomposition of $\mathcal{H}$ denoted by $\mathcal{E}=(\cT',\{H_v\}_{v\in V(\cT')},\{\iota_e\}_{e\in E(\cT')})$ so that for every $n\geq 0$ one of the following holds
\begin{itemize}
\item either all $C$-inclusions in $(\iota_e)_{e\in E_n(\cT')}$ are bijective, or
\item $o(e)\neq o(e')$ whenever $e,e'$ are different edges in $E_n(\cT')$ and there exists $e_{\ast}\in E_n(\cT')$ such that \begin{itemize}
\item $j_{e_{\ast}}$ is an elementary $C$-inclusion
\item $j_e$ is bijective for $e\in E_n(\cT')\setminus \{e_{\ast}\}$
\end{itemize}
\end{itemize}

\end{proposition}
In this case we say that $\mathcal{E}$ an \emph{elementary decomposition} of $\mathcal{H}$.

\begin{remark}\label{r.includecomp} It is straightforward to check that if a forest of $C$-graphs $\mathcal{H}$ is realized in a tower $\mathbb{U}=\left\{q_n:\Gamma_{n+1}\to \Gamma_n\right\}$ and $\sigma:\N\to\N$ is a strictly increasing map then, $\mathcal{H}_\sigma$ is included in the tower $\mathbb{U}_{\sigma}=\left\{q^\sigma_{n}:\Gamma_{\sigma(n+1)}\to \Gamma_{\sigma(n)}\right\}$ where $q^\sigma_{n}=q_{\sigma(n)}\circ\ldots\circ q_{\sigma(n+1)-1}$.

\end{remark}

\begin{lemma}[Main Lemma]\label{l.main} Consider an elementarily decomposable forest \\$\mathcal{H}=(\mathcal{T},\{H_v\}_{v\in V(\mathcal{T})},\{\iota_e\}_{e\in E(\mathcal{T})})$. Assume that there exists $\Gamma_0$ a finite covering space of the figure eight $C$-graph together with: \begin{itemize} 
\item $\{G_v:v\in V_0(\mathcal{T})\}$ a disjoint family of subgraphs of $\Gamma_0$ and
\item $\{f_v:G_v\to H_v:v\in V_0(\mathcal{T})\}$ a family of collapses.
\end{itemize}
Then $\mathcal{H}$ can be realized in a tower.

\end{lemma}

\begin{proof} Let $\mathcal{H}=(\mathcal{T},\{H_v\}_{v\in V(\mathcal{T})},\{\iota_e\}_{e\in E(\mathcal{T})})$ be an elementarily decomposable forest and $\mathcal{E}=(\cT',\{H_v\}_{v\in V(\cT')},\{\iota_e\}_{e\in E(\cT')})$ an elementary decomposition of $\mathcal{H}$ given by Proposition \ref{p.decomposition}. By Remark \ref{r.includecomp}, in order to show that $\mathcal{H}$ can be realized in a tower, it is enough to show it for $\mathcal{E}$. 

A realization of $\mathcal{E}$ up to level $n$ is defined as the following data:
\begin{itemize}
\item finite coverings $\{q_i:\Gamma_{i+1}\to\Gamma_i:i=0,\ldots,n-1\}$;
\item for each $i\leq n$, a family $\{G_v:v\in V_i(\mathcal{T}')\}$ of disjoint subgraphs of $\Gamma_i$;
\item for each $i\leq n-1$, a family of $C$-inclusions $\{j_e:G_{o(e)}\to G_{t(e)}:e\in E_i(\cT')\}$ satisfying $q_i\circ j_e=id$ and 
\item a family of collapses $\{f_{v}:G_v\to H_v:v\in E_i(\mathcal{T}'),i\leq n\}$ satisfying\\ $f_{t(v)}\circ j_e=i_e\circ f_{o(e)}$ for every $e\in E_i(\cT')$ with $i\leq n-1$.
\end{itemize}

We are going to prove that any realization of $\mathcal{E}$ up to level $n$ can be extend to a realization up to level $n+1$. Since the hypothesis of the Lemma implies that $\mathcal{E}$ can be realized up to level $0$, the Lemma will follow by induction. 

For this, assume that $\mathcal{E}$ can be realized up to level $n$. By Proposition \ref{p.decomposition} we need to distinguish in two cases.\\

\noindent\emph{Case 1. For every $e\in E_n(\cT')$, $j_e$ is bijective}.  

In that case, let $m_v=\#\{e\in E_n(\cT'):o(e)=v\}$. Notice that Case 1, we just need to construct a finite covering $q_n:\Gamma_{n+1}\to\Gamma_n$ containing $m_v$ disjoint $(1:1)$ lifts of $G_v$ for each $v\in V_n(\cT')$. The existence of such covering follows directly from Lemma \ref{l.multiplelift} which allows to replicate these subgraphs. Denote $\{j_e:G_{o(e)}\to G_{t(e)};e\in E_n(\cT')\}$ the given family of lifts under $q_n$. Finally define $f_{t(e)}:=j_e\circ f_{o(e)}\circ (q_n|_{G_{t(e)}})$ for every $e\in E_n(\cT')$.\\

\emph{Case 2. For every pair of different edges $e,e'$ in $E_n(\cT')$, we have $o(e)\neq o(e')$ and moreover there exists} $e_{\ast}\in E_n(\cT')$ \textit{such that} \begin{itemize}
\item $j_{e_{\ast}}$ \textit{is an elementary $C$-inclusion}
\item $j_e$ \textit{is bijective for} $e\in E_n(\cT')\setminus \{e_{\ast}\}$
\end{itemize}
Let $G_0:=\bigsqcup_{e\in E_n(\cT')\setminus \{e_{\ast}\}}G_{o(e)}$. In  Case 2, we must construct a finite covering $q_n:\Gamma_{n+1}\to\Gamma_n$ together with \begin{itemize}
\item $G_{t(e_{\ast})}$ and $\hat G_0:=\bigsqcup_{e\in E_n(\cT')\setminus \{e_{\ast}\}}G_{t(e)}$ subgraphs of $\Gamma_{n+1}$
\item $j_{e_\ast}:G_{o(e_\ast)}\to G_{t(e_{\ast})}$ and $j_0:G_0\to\hat G_0$, $(1:1)$ lifts under $q_n$
\item a collapse $f_{t(e_\ast)}: G_{t(e_\ast)}\to H_{t(e_{\ast})}$
\end{itemize}
that satisfy $f_{t(e_\ast)}\circ j_{e_\ast}=\iota_{e_\ast}\circ f_{o(e_\ast)}$ and that $j_0$ is a $(1:1)$ lift. This follows directly from Lemma \ref{l.bblock} setting:
\begin{itemize} 
\item $G_0:=G_0$ and $G:=G_{o(e_{\ast})}$, 
\item $f:=f_{o(e_{\ast})}$ and $\hat f=f_{t(e_\ast)}$,
\item $\iota:=\iota_{e_{\ast}}$ and $j=j_{e_\ast}$ 
\end{itemize} 
This finishes the proof of the Lemma. 
\end{proof}

\section{Proof of Theorem \ref{t.grafos}}\label{s.proof} Recall that a pair $(\cE_0,\cE)$ satisfies condition $(\ast)$ if $\cE_0$ is a closed subset of $\cE$, $\cE$ is a compact and totally disconnected metrizable space and $\cE_0$ contains the isolated points of $\cE$. 

First, we prove the following weak version of Theorem \ref{t.grafos}:

\begin{proposition}\label{p.infinite} There exists a tower $\mathbb{U}=\{q_n:\Gamma_{n+1}\to\Gamma_n\}$ whose inverse limit $\mathcal{M}$ satisfies:
\begin{itemize}
\item its generic leaf is a tree;
\item given any pair $\tau=(\cE_0,\cE)$ satisfying condition $(\ast)$, there exists a leaf of $\mathcal{M}$ whose ends pair is equivalent to $\tau$.
\end{itemize}\end{proposition}
Notice that the only missing classifying triples in Proposition \ref{p.infinite} are those of the form $(g,\emptyset,\cE)$. Also notice that by condition $(\ast)$, those classifying triples must satisfy that $\cE$ is perfect and therefore homeomorphic to a Cantor set. After proving Proposition \ref{p.infinite}, we will show how to modify the construction in order to realize also these countably many missing triples. 

In order to prove Proposition \ref{p.infinite} we need another Proposition whose proof will be postponed until the next section.
\begin{proposition}\label{p.forestuniv} There exists an elementarily decomposable forest of $C$-graphs $\mathcal{F}=(\mathcal{T},\{H_v\}_{v\in V(\mathcal{T})},\{\iota_e\}_{e\in E(\mathcal{T})})$ satisfying:
\begin{itemize}
\item $V_0(\cT)$ consists of three vertices, that we denote by $v_1,v_2$ and $v_3$; moreover, $H_{v_1}$ is an $h_2$-piece, $H_{v_2}$ is an $s$-piece and $H_{v_3}$ is an $h_4$-piece;
\item $\iota_e(H_{o(e)})\subseteq \Int(H_{t(e)})$ for every $e\in E(\cT)$;
\item for every pair $\tau=(\cE_0,\cE_1)$ satisfying condition $(\ast)$, there exists $\alpha\in\cE(\cT)$ such that $(\cE_0(H^\alpha),\cE(H^\alpha))$ is equivalent to $\tau$. 
\end{itemize}
\end{proposition}

\begin{proof}[Proof of Proposition \ref{p.infinite} using \ref{p.forestuniv}: ]Consider the forest of $C$-graphs $$\mathcal{F}=(\mathcal{T},\{H_v\}_{v\in V(\mathcal{T})},\{\iota_e\}_{e\in E(\mathcal{T})})$$ constructed in Proposition \ref{p.forestuniv}. Notice that $V_{0}(\cT)$ consists of three vertices, that we denote $\{v_{1},v_2,v_3\}$, and that $H_{v_1}$ is an $h_2$-piece, $H_{v_2}$ is an $h_4$-piece and $H_{v_3}$ is an $s$-piece. In order to realize $\mathcal{F}$ in a tower using Lemma \ref{l.main}, we need to construct a finite covering of the figure eight $C$-graph $\Gamma_0$ together with\begin{itemize}
\item $G_{v_1},G_{v_2}$ and $G_{v_3}$ disjoint $C$-subgraphs of $\Gamma_0$ and
\item collapses $f_i:G_{v_i}\to H_{v_i}$ for $i=1,2,3$. 
\end{itemize}

\begin{figure}[h!]
\centering
\includegraphics[scale=0.12]{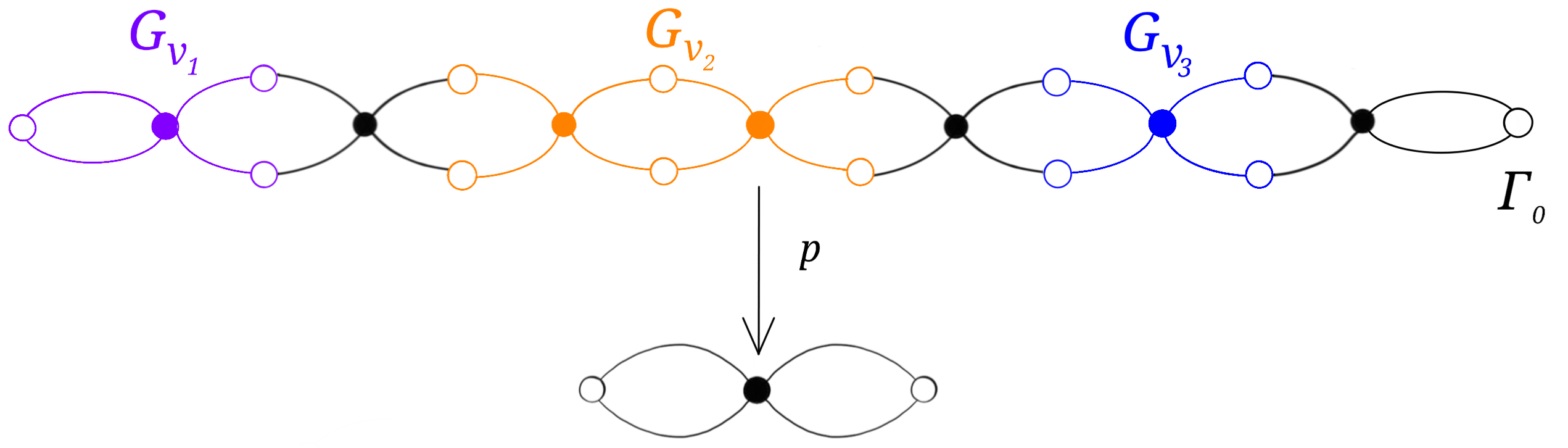}
\caption{A $7$-fold covering with the three basic pieces needed to initialize our induction: the purple subgraph collapses onto an $h_2$-piece, the orange one onto a $h_4$-piece, and the blue one is an $s$-piece.}\label{fig:initouille}
\end{figure}

This can be done with a $7$-fold covering graph over $\Gamma_0$ as shown in Figure \ref{fig:initouille}. Then, we can apply Lemma \ref{l.main} and realize $\mathcal{F}$ inside a tower $\mathbb{U}$ as desired. Let $\mathcal{M}$ denote the inverse limit of $\mathbb{U}$. Since $\mathcal{F}$ satisfies that $\iota_e(H_{o(e)})\subseteq \text{Int}(H_{t(e)})$ for every $e\in E(\cT)$, Proposition \ref{p.include} implies the existence of a family of leaves $\{\cL^\alpha:\alpha\in\cE(\cT)\}$ of $\mathcal{M}$ verifying that $(\cE_0(H^\alpha),\cE(H^\alpha))$ is equivalent to $(\cE_0(\cL^\alpha),\cE(\cL^\alpha))$ for every $\alpha\in\cE(\cT)$. Therefore, by Proposition \ref{p.forestuniv}, all equivalence classes of end pairs satisfying condition $(\ast)$ are realized in $\{(\cE_0(\cL^\alpha),\cE(\cL^\alpha)):\alpha\in\cE(\cT)\}$. 

Finally, since there exists a leaf $\cL^\alpha$ with $\cE_0(\cL^\alpha)=\emptyset$ we can apply Proposition \ref{p.generic_tree} to show that the generic leaf of $\mathcal{M}$ is a tree. 
\end{proof}

\paragraph{\textbf{Modifying the construction to realize the missing triples}}
We proceed to show how to modify the construction of Proposition \ref{p.infinite} in order to (also) include leaves realizing the classifying triples $$\{(g,\emptyset,K):g> 0\text{ and }K\text{ a Cantor set}\}$$ 

For this we need to introduce some definitions and notations. Denote by $T_\ast$ the tree with a unique $b_1$-vertex that is obtained by glueing $s$-pieces. Given a $C$-graph $G$, define $T(G)$ as the $C$-graph obtained by glueing copies of $T_\ast$ at each $b_1$-vertex of $G$ (see Figure \ref{fig:tidgi}). Note that $\beta_1(T(G))=\beta_1(G)$. Finally, define $$T(G,r)=\{w\in T(G):\dist_{T(G)}(G,w)\leq 2r\}.$$ 
Notice there exist natural $C$-inclusions $\iota_r:T(G,r)\to T(G,r+1)$ for every $r\geq 1$ and that $T(G)$ is isomorphic to $\underrightarrow{\lim}\{\iota_{r}:T(G,r)\to T(G,r+1)\}$. 

\begin{figure}[h!]
\centering
\includegraphics[scale=0.1]{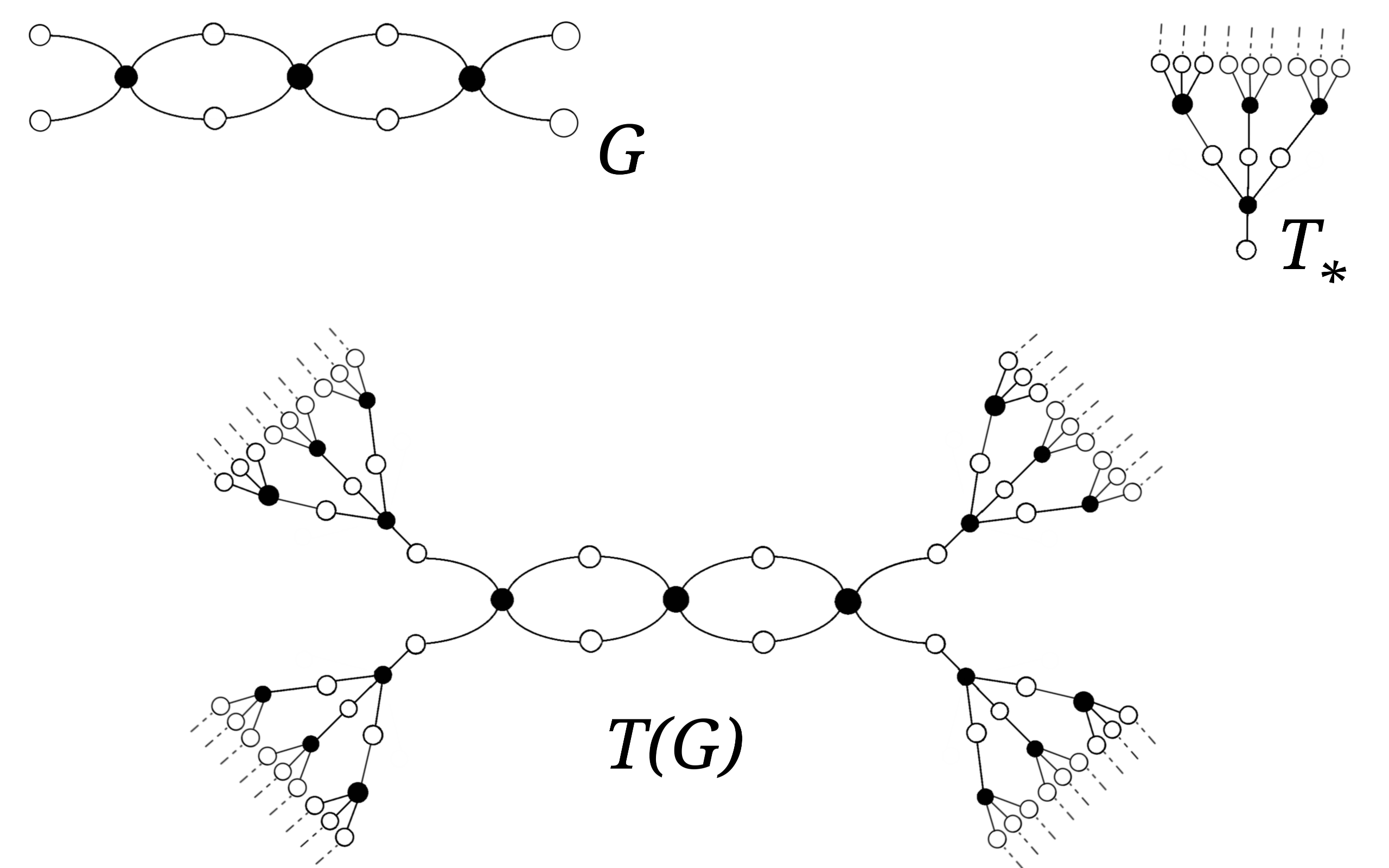}
\caption{The tree $T_\ast$, a $C$-graph $G$ with $\beta_1(G)=2$ and four $b_1$-vertices, and the corresponding infinite graph $T(G)$.}\label{fig:tidgi}
\end{figure}

Roughly speaking, our idea to modify the construction in Lemma \ref{l.main} while also including $C$-subgraphs of the form $T(G,r)$ and lifts of the form $\iota_r:T(G,r)\to T(G,r+1)$ inside our tower. We proceed with our construction. 

First we show the following result
\begin{proposition}[Including graphs with finite dimensional homology]\label{p.include_finite_homo} There exist
\begin{itemize}
\item a tower $\mathbb{U}=\{q_n:\Gamma_{n+1}\to\Gamma_n\}$ realizing the forest $\mathcal{F}$ of Proposition \ref{p.forestuniv} via a $C$-subgraph forest $\mathcal{S}=(\cT,\{G_v\}_{v\in V(\cT)},\{j_e\}_{e\in E(\cT)});$
\item for each $n\geq 1$, a disjoint family of $C$-subgraphs of $\Gamma_n$ denoted\\ $\{G_{i,n}:1\leq i\leq n\}$ such that:
\begin{itemize}
\item $G_{i,n}$ is disjoint from $G_v$ for every $v\in V_n(\cT)$ and $i\leq n$;
\item $\beta_1(G_{i,n})=i$ for $i\geq 1$;
\item $G_{i,n+1}=T(G_{i,n},1)$ for $1\leq i\leq n$;
\end{itemize}
\item a family of $C$-inclusions $\{j_{i,n}:G_{i,n}\to G_{i,n+1}\}$ such that 
\begin{itemize}
\item $q_{n}\circ j_{i,n}=\Id$
\item $j_{i,n}(G_{i,n})\subseteq\Int(G_{i,n+1})$
\end{itemize}
\end{itemize}
\end{proposition}

To prove that such a tower exists, we use the following Lemma which is a variant of Lemma \ref{l.bblock}, and whose proof is left to the reader.
\begin{lemma}\label{l.variationbblock}Consider $\Gamma$ a finite covering space of the figure eight $C$-graph, $G,G_0$ disjoint $C$-subgraphs of $\Gamma$ and $m\in\N$. Then, there exists a finite covering $p:\hat\Gamma\to\Gamma$ and $\hat G,\hat G_0,\hat F$ disjoint $C$-subgraphs of $\hat\Gamma$ such that \begin{itemize}
\item $\beta_1(\hat F)=m$;
\item $\hat G_0$ is a lift of $G_0$;
\item $\hat G$ is isomorphic to $T(G,	1)$;
\item there exists $j:G\to\hat G$ such that 
\begin{itemize}
\item $p\circ j=id$ and
\item $j(G)\subseteq \Int(\hat G)$.
\end{itemize}
\end{itemize}
\end{lemma}

\begin{proof}[Proof of Proposition \ref{p.include_finite_homo}] 

Following the proof and notations of Lemma \ref{l.main}, we say that the inductive property (IP) is satisfied up to level $k$ and we denote it by (IP)$_k$ if there exist

\begin{itemize}
\item a realization of $\mathcal{F}$ up to level $k$ denoted $\{q_i:\Gamma_{i+1}\to\Gamma_i;i=0,\ldots,k-1\}$ (recall the definition of realization given in \S \ref{ss.forests_inclusion_realization});
\item a family of disjoint $C$-subgraphs $\{G_{i,n}:1\leq i\leq n\leq k\}$ with $G_{i,n}\subseteq \Gamma_n$ satisfying all the properties stated above and
\item a family of $C$-inclusions $j_{i,l}:G_{i,l}\to G_{i,l+1}$ with $1\leq i\leq l\leq k$, satisfying all the properties stated above. 
\end{itemize}

Now, we prove that (IP)$_k$ implies (IP)$_{k+1}$. To do so, let $\mathcal{G}_k=\bigcup_{i\leq k}G_{i,k}$. Now, we are going to proceed as in Lemma \ref{l.main} but with a slight variation. Consider a covering $q:\Gamma^\ast_{k+1}\to \Gamma_k$ extending the realization of $\mathcal{F}$ one more floor so that, in addition, $\Gamma^\ast_{k+1}$ contains a $C$-subgraph $\mathcal{G}'_k$ which is a $(1:1)$ lift of $\mathcal{G}_k$. Then, apply Lemma \ref{l.variationbblock} with $\Gamma=\Gamma_{k+1}^\ast$, $G_0=\bigcup_{v\in V_{k+1}(\cT)}G_v$, $G=\mathcal{G}'_k$ and $k=m$ to construct a finite covering $p:\Gamma_{k+1}\to\Gamma_{k+1}^\ast$ together with all the $C$-subgraphs and $C$-inclusions as stated in the Lemma. 

We claim that $q_{k+1}:=q\circ p:\Gamma_{k+1}\to\Gamma_k$ is the desired covering. To see this, compose the $C$-inclusion associated to the covering $q$ with those associated to $p$. It is straightforward to check that $\Gamma_{k+1}$ contains all the desired $C$-subgraphs and therefore that $\{q_i:\Gamma_{i+1}\to\Gamma_i;i=0,\ldots,k\}$ satisfies (IP)$_{k+1}$. Finally, by induction, we obtain the desired tower $\mathbb{U}$.
\end{proof}
 
\begin{proof}[End of proof of Theorem \ref{t.grafos}]
To check that $\mathbb{U}$ satisfies the two conditions required in Theorem $\ref{t.grafos}$, let $\mathcal{M}$ denote the inverse limit of the tower $\mathbb{U}$. Since $\mathbb{U}$ realizes $\mathcal{F}$, $\mathcal{M}$ realizes all classifying triples satisfying condition $(\ast)$ with infinite dimensional homology. To check that classifying triples satisfying condition $(\ast)$ with finite dimensional homology are realized, note that for every $k\geq 1$, $T(G_{k,k})$ is isomorphic to the direct limit $\underrightarrow{\lim}\{j_{k,l}:G_{k,l}\to G_{k,l+1};l\geq k\}$. Therefore, we can argue as in Proposition \ref{p.iso} to show the existence of leaves isomorphic to $T(G_{k,k})$, for every $k\geq 1$. Since $\beta_1(T(G_{k,k}))=k$ and $\cE(T(G_{k,k}))$ is a Cantor set, this finishes the proof of Theorem \ref{t.grafos}.
\end{proof}

\section{Proof of Proposition \ref{p.forestuniv}}\label{s.forestuniv} Recall that we endow $C$-graphs with the path distance where all edges have length one. Let $\text{dist}$ denote this distance and $B_{G}(v,r)=\{w\in G:\dist(v,w)\leq r\}$. We say that $(G,v)$ is a \emph{pointed $C$-graph} if $v\in G$ \emph{is not of boundary type}. 
We omit the pointing from the notation unless it creates confusion. In this spirit, when $(G,v)$ is a pointed $C$-graph we write $B_G(n)$ instead of $B_G(v,n)$. Finally, denote $[(G,v)]$ the class of $(G,v)$ up to pointing-preserving isomorphisms.

\subsection{The construction} In order to construct our forest of $C$-graphs with the desired limits we take the reverse path. First we define a family of $C$-graphs that we want to realize as limits and then we construct the forest of $C$-graphs realizing the family as limits. We proceed to define this family.
\paragraph{\textbf{The family $\mathcal{C}$}}Say that a pointed $C$-graph $(G,v)$ belongs to the family $\mathcal{C}$ if
\begin{enumerate}
\item $\partial G=\emptyset$; and
\item $B_G(v,2n+1)$ is obtained from $B_G(v,2n-1)$ by adding a disjoint union of 
\begin{itemize} 
\item $h_4$-pieces meeting $\partial B_G(v,2n-1)$ at exactly two boundary vertices and,
\item $s$ and $h_2$-pieces meeting $\partial B_G(v,2n-1)$ at exactly one boundary vertex.
\end{itemize}

\end{enumerate}

\begin{remark}\label{r.familyc}Note that, since pointings are not of boundary type, the balls $B_G(2n+1)$ are $C$-subgraphs. Also, by Condition $(2)$ we have that the $C$-inclusions $$\iota:B_{G}(2n-1)\to B_G(2n+1)$$ are elementarily decomposable. Finally, by definition, we have that $B_G(2n+1)$ strictly contains $B_G(2n-1)$ which in particular implies that every graph $G$ in $\cG$ is infinite. 
\end{remark}

\begin{remark}\label{rem.CgraphC} It is easy to check that the $C$-graphs illustrated in Figures \ref{fig:evouille} and \ref{fig:forouille} belong to the family $\cC$ (independently on the pointing).
\end{remark}

\paragraph{\textbf{The construction of $\cT$}}We proceed to construct the underlying forest $\mathcal{T}$, for this define $$V_n(\mathcal{T})=\Bigg\{\bigg[B_{G}(2n+1)\bigg]:G\in\mathcal{C},n\geq0\Bigg\}$$ 
where, as we recall, we are considering pointed $C$-graphs up to pointing-preserving isomorphisms. Clearly $V_n(\mathcal{T})$ is finite for every $n$. 
 Moreover $V_0(\cT)$ consists of three vertices $v_1,v_2$ and $v_3$ corresponding respectively to an $h_2$-piece, an $h_4$-piece and an $s$-piece.
On the other hand we define that $$([B_1],[B_2])\in E_n(\mathcal{T})\subseteq V_n(\mathcal{T})\times V_{n+1}(\mathcal{T})$$ if there exists a $C$-inclusion $\iota:B_1\to B_2$ preserving the pointing. Notice that, since we are considering $C$-inclusions preserving the pointing, for every $[B]\in V_n(\cT)$ with $n>0$ there exists exactly one edge $e\in E(\cT)$ with $t(e)=[B]$. Therefore, $\cT$ has no cycle. 

\paragraph{\textbf{The construction of }$\mathcal{F}$}First we construct a forest of pointed $C$-graphs with pointing preserving $C$-inclusions. Given $[B]\in V(\mathcal{T})$ define $H_{[B]}$ as any representative of $[B]$, and given 

$$e=([B_1],[B_2])\in E(\mathcal{T}),$$ 
define $\iota_e$ as any pointing preserving $C$-inclusion from $H_{[B_1]}$ to $H_{[B_2]}$. Then, we define our forest of $C$-graphs as $$\mathcal{F}=(\mathcal{T},\{H_{[B]}\}_{[B]\in V(\mathcal{T})},{\{ \iota_e\}}_{e\in E(\mathcal{T})})$$ where we forget the pointings of the $\{H_{[B]}\}_{[B]\in V(\cT)}$. 

Notice that by Remark \ref{r.familyc}, the forest of $C$-graphs $\mathcal{F}$ is elementarily decomposable. Also note that, since elements of $\mathcal{C}$ have empty boundary it holds that $$\iota_e(H_{o(e)})\subseteq \text{Int}(H_{t(e)}) \text{ for every }e\in E(\cT).$$ 

\begin{remark} By construction we have that if $G\in\mathcal{C}$ then $([B_G(2n+1)])_{n\in\N}$ is a path in $\cT$ converging to an end $\alpha\in\cE(\mathcal{T})$ with $H^{\alpha}$ isomorphic to $G$. In other words: \begin{center}\emph{All the elements in the family $\mathcal{C}$ are realized as limits of the forest $\mathcal{F}$.}\end{center}
\end{remark}

In order to finish the proof of Proposition \ref{p.forestuniv} it remains to show that the ends pairs of the elements of $\mathcal{C}$ realize all pairs satisfying condition $(\ast)$.

\subsection{The ends pairs of elements in $\mathcal{C}$}
First, we show that every pair $(K_0,K_1)$ satisfying condition $(\ast)$ with $K_1$ infinite, is realized as an end pair of an element in $\mathcal{C}$. That is the content of the following Proposition:

\begin{proposition}
\label{p.trees_star}
For every pair $(K_0,K_1)$ satisfying condition $(\ast)$ with $K_1$ infinite, there exists a $C$-graph $G\in\cC$ such that $(\cE_0(G),\cE(G))$ is equivalent to $(K_0,K_1)$. Therefore, every such pair is realized as an ends pair of a limit of $\mathcal{F}$.
\end{proposition}
In order to prove Proposition \ref{p.trees_star} we need to introduce some definitions and notations. 

\paragraph{\textbf{Adapted sequences of partitions}}

If $X$ is a set, $\xi\dans\mathcal{P}(X)$ a partition and $x\in X$, we define $\xi(x)=A$ where $x\in A$ and $A\in\xi$. Given a set $X$ with partitions $\xi_1,\xi_2$, we say that $\xi_2$  \emph{is finer than} $\xi_1$ if $\xi_2(x)\dans\xi_1(x)$ for every $x\in X$. In this case we note $\xi_1\prec\xi_2$.

Let $(K_0,K_1)$ be a pair satisfying condition $(\ast)$ with $K_1$ infinite. Recall that by definition this means that isolated points of $K_1$ belong to $K_0$. Let $\xi_1,\xi_2,\ldots$ be a sequence of partitions of $K_1$. We say that $(\xi_n)_{n\in\N}$ is \emph{adapted} to the pair $(K_0,K_1)$ if it satisfies the following properties
\begin{enumerate}[(\text{A} 1)]
\item\label{c1} $\xi_i$ is a finite partition by clopen sets for every $i\in\N$;

\item\label{c2} $\#\xi_0=1$ and $\#\xi_1=2$ or $4$;

\item\label{c3} $\xi_{i}\prec\xi_{i+1}$ for every $i\in\N$ (i.e.  $\xi_{i+1}$ refines $\xi_i$);

\item\label{c4} for every $i\geq 1$, if $A\in\xi_i\setminus\xi_{i+1}$ then, there exist three different and non-empty elements $B_1,B_2,B_3\in\xi_{i+1}$ such that $A=B_1\sqcup B_2\sqcup B_3$;

\item\label{c5} for every $i\geq 1$ and  $A\in\xi_i$ we have that $A\cap K_0\neq\emptyset$ if and only if $A\in\xi_{i-1}\cup\xi_{i+1}$;
\item\label{c6} given two distinct points $x,y\in K_1$ there exists $i\in \N$ such that $y\notin\xi_i(x)$ (i.e. the sequence separates points).
\end{enumerate}

\begin{remark}\label{r.particionfinita} Notice that conditions (A~\ref{c2}, \ref{c3}, \ref{c4}) imply that $\#\xi_i\leq 4.3^{i-1}$\end{remark}

\paragraph{\textbf{From adapted sequences to pointed $C$-graphs}} Consider a pair $(K_0,K_1)$ satisfying condition $(\ast)$ with $K_1$ infinite, and $\xi=(\xi_i)_{i\in\N}$ a sequence of partitions adapted to the pair $(K_0,K_1)$. We will construct a pointed $C$-graph $G_{\xi}\in\mathcal{C}$ satisfying that $(\cE_0(G_{\xi}),\cE(G_{\xi}))$ is equivalent to $(K_0,K_1)$. This construction together with the following Lemma (whose proof we leave to the appendix, see Section \ref{s.Appendix}) will finish the proof of Proposition \ref{p.trees_star}.

\begin{lemma}\label{l.cantor} Let $(K_0,K_1)$ be a pair satisfying condition $(\ast)$ with $K_1$ infinite. Then there exists $(\xi_i)_{i\in\N}$, a sequence of finite partitions adapted  to $(K_0,K_1)$. 
\end{lemma}

\begin{figure}[h!]
\centering
\includegraphics[scale=0.11]{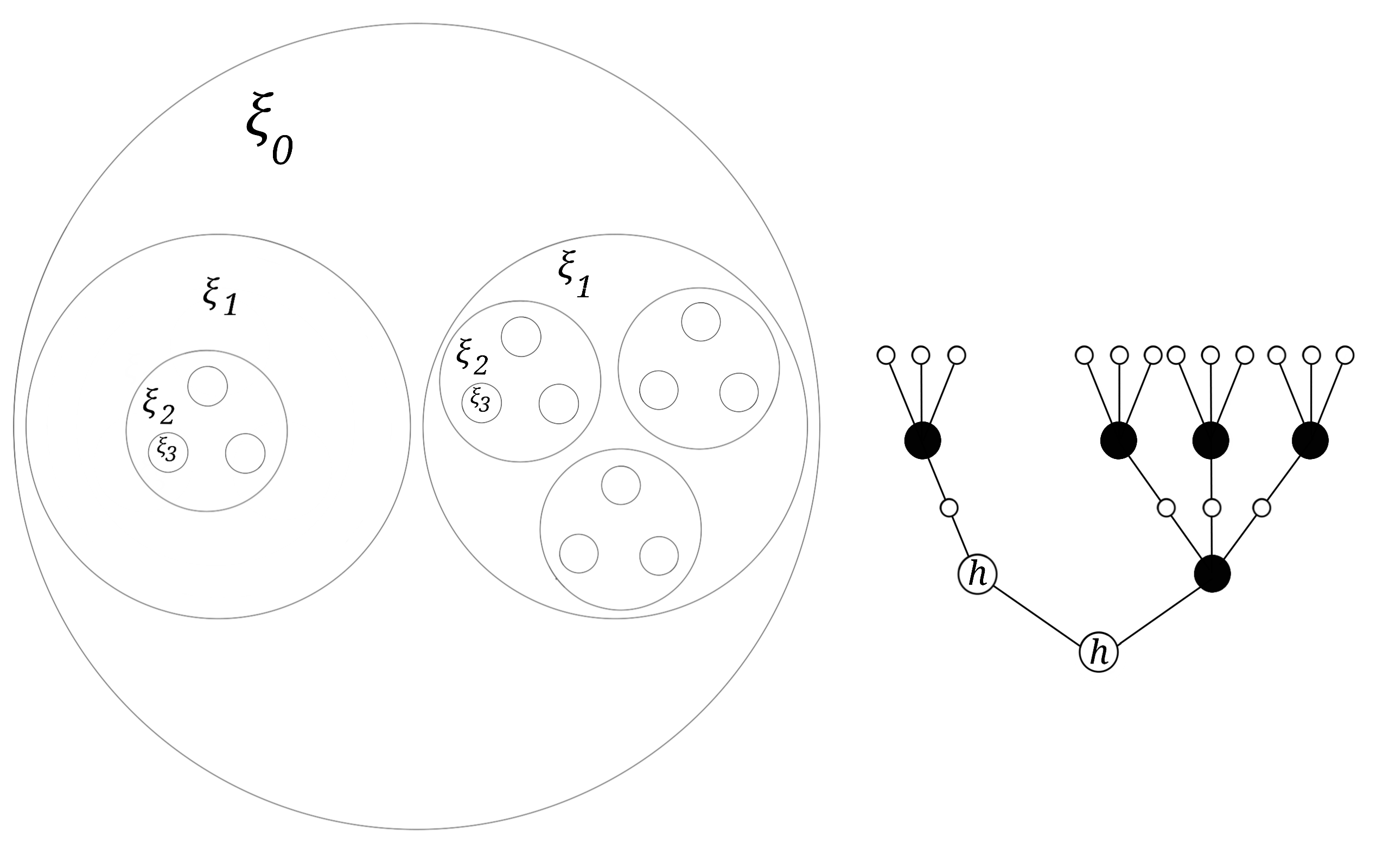}
\caption{An adapted sequence of partition (on the left) and the associated $C$-graph $G_\xi$ (on the right).}\label{fig:adptouille}
\end{figure}

\begin{proof}[Proof of Proposition \ref{p.trees_star}] Let $(K_0,K_1)$ be a pair satisfying condition $(\ast)$ with $K_1$ infinite and $\xi=(\xi_i)_{i\in\N}$ be a sequence of finite partitions adapted to $(K_0,K_1)$. We build a $C$-graph $G_\xi\in\cC$ whose ends pair is equivalent to $(K_0,K_1)$ (see Figure \ref{fig:adptouille}). First define a tree $T_{\xi}$ as:
\begin{itemize}
\item $V(T_{\xi})=\bigsqcup_{i\in\N} \xi_i$ (condition (A \ref{c1}) implies that each $\xi_i$ is a finite set)
\item $(A,B)\in E(T_{\xi})\subseteq V(T_{\xi})\times V(T_{\xi})$ if and only if there exists $n\in\N$ such that $A\in\xi_n,B\in\xi_{n+1}$ and $B\subseteq A$ (conditions (A \ref{c2}, \ref{c3}) imply that every vertex has valency at least 2).
\end{itemize}
Notice that by conditions (A \ref{c2}, \ref{c4}), the valency of vertices is either two or four.  

To transform $T_{\xi}$ into a pointed $C$-graph $G_{\xi}\in\mathcal{C}$, consider the pointing of $T_{\xi}$ at $K_1\in\xi_0$, add boundary vertices in edges midpoints and label other vertices according to their valencies: valency two vertices of non-boundary  type become $h_2$-vertices and valency four  vertices become $s$-vertices. This finishes the construction of $G_{\xi}$.

Proposition \ref{p.trees_star} now follows from the next lemma.
\end{proof}

\begin{lemma}\label{l.equiv} The pair $(\cE_0(G_\xi),\cE(G_\xi))$ is equivalent to $(K_0,K_1)$. 
\end{lemma}
\begin{proof}[Proof of Lemma \ref{l.equiv}]

Take $x\in K_1$ and consider the sequence $(\xi_i(x))_{i\in\N}$. Note that this sequence defines an infinite ray in $T_\xi$, which represents an end. See Figure \ref{fig:adptouille}. Define the map $\varphi:K_1\to\cE(G_{\xi})$ sending each $x\in K_1$ to the ray represented by $(\xi_i(x))_{i\in\N}$. We proceed to show that $\varphi$ induces the desired equivalence of pairs.

The injectivity of $\varphi$ comes from condition (A \ref{c6}). The surjectivity comes from the fact that decreasing sequences of nonempty compact metric spaces have nonempty intersections. 

Given $A\in\xi_n$ a vertex of $G_{\xi}$ not of boundary type we can define $W_{A}$ as the set of ends represented by embedded rays that start at the pointing of $G_{\xi}$ and pass through $A$. Notice that $$\{W_A:A\text{ vertex of non-boundary type } \}$$ is a basis of the topology of $\cE(G_{\xi})$. Then, since $\varphi^{-1}(W_A)=\{x:x\in A\}$ is a clopen set we get that $\varphi$ is continuous as desired. 

It remains to prove that $\varphi$ induces a bijective correspondence between $K_0$ and $\cE_0(G_{\xi})$. Note first that $G_\xi$ is topologically a tree. Hence an end of $\xi$ belongs to $\cE_0(G_\xi)$ if and only if it is accumulated by $h$-vertices.

Let $x_0\in K_0$ and $\xi_n(x_0)$ the sequence of elements of $\xi_n$ containing $x_0$. These are clopen sets and the sequence $(\xi_n)_{n\in\N}$ separates points, so $\{\xi_n(x_0):n\in\N\}$ forms a neighbourhood basis of $x_0$. Using condition (A \ref{c5}) above we see that for infinitely many $n\in\N$, $\xi_n(x_0)=\xi_{n+1}(x_0)$. So the ray defined by the sequence $(\xi_n(x_0))_{n\in\N}$ has infinitely many $2$-valent vertices which are $h$-vertices. This proves that $\fhi(x_0)\in\cE_0(G_\xi)$.

Now we consider $x\in K_1\moins K_0$ and look at the ray defined by $(\xi_n(x))_{n\in\N}$. Since $x\notin K_0$, which is closed inside $K_1$, there exists a neighbourhood $V$ of $x$ such that $V\cap K_0=\vide$ and therefore, there exists $n_0\in\N$ such that $\xi_{n_0}(x)\dans V$. Let $C^+$ be the cone of $G_{\xi}$ consisting of the union of the connected components of $G_\xi\moins\{\xi_{n_0}(x)\}$ which don't contain the pointing. There are three of such components by definition because $K_0\cap\xi_{n_0}(x)=\vide$ (this is implied by conditions (A \ref{c4}, \ref{c5})). The same argument shows that every (non-boundary) vertex inside $C^+$ has valency $4$ (since for every $y\in\xi_{n_0}(x)$ and $n>n_0$, $\xi_n(y)\cap K_0=\vide$). This means in particular that the end represented by the ray  $(\xi_n(x))_{n\in\N}$ is not accumulated by $h$-vertices (which are $2$-valent). Hence $\fhi(x)\in\cE(G_\xi)\moins\cE_0(G_\xi)$. This finishes the proof of the Lemma.
\end{proof}

\paragraph{\textbf{Realizing finite ends pairs with elements of $\mathcal{C}$}} In order to finish the proof of Proposition \ref{p.forestuniv} it remains to show that finite ends pairs satisfying condition $(\ast)$ are realized as ends pairs of elements of $\mathcal{C}$. Examples of $C$-graphs with $1,2$ and $3$ ends are shown in Figures \ref{fig:evouille} and \ref{fig:forouille}. In order to construct $C$-graphs with arbitrary number of ends we define an inductive procedure which, from a given $C$-graph with finitely many ends produces a new one with $2$ more ends. For this, we define the $h_2$-\emph{ray} as the one-ended $C$-graph with exactly one valency $1$ boundary vertex which is obtained by concatenating infinitely many $h_2$-pieces. Also, we define the $h_2$-\emph{trident} as the $C$-graph obtained by gluing $3$ $h_2$-rays at an $s$-piece. Note that if a $C$-graph has finitely many ends, the result of the substitution of an $h_2$-ray by an $h_2$-trident increases by $2$ the number of ends. Some examples are shown in Figures \ref{fig:evouille} and \ref{fig:forouille}. The first one shows how to realize $C$-graphs which are topological trees with an even number of ends. The second one shows how to treat $C$-graphs with an odd number of ends. Such a graph is not a topological tree, and one special end is approximated by  vertices of $h_4$-type. As noticed in Remark \ref{rem.CgraphC} these graphs belong to the family $\cC$ and their ends pairs satisfy condition $(\ast)$.

\section{Appendix}
\label{s.Appendix}

\subsection{Proof of Corollary \ref{corollary}}\label{s.corollary}

Consider the lamination $\cL$ constructed in Theorem \ref{maintheorouille}. By construction it comes with a structure of bundle $\Pi_0:\cL\to\Sigma_0$ whose fiber is a Cantor set $K$. Let $D\dans \Sigma_0$ be a small open disc trivializing the bundle so that $U=P_0^{-1}(D)$ is homeomorphic to $D\times K$ and $\partial U$ is homeomorphic to $\partial D\times K$. Define $\mathcal{M}:=\cL\setminus U$ and $\partial \mathcal{M}:=\partial U$. Consider $T$ a copy of the one holed torus and define $\mathcal{V}:=T\times C$ and $\partial \mathcal{V}:=\partial T\times K$. 

We define $\cL'$ by (continuously) identifying the boundaries of $\mathcal{M}$ and $\mathcal{V}$ so that $\partial D\times\{x\}$ identifies with $\partial T\times \{x\}$ for every $x\in K$. Given $x\in\mathcal{M}(=\cL\cap\cL')$ let $\cL_x$ and $\cL_x'$ denote the leaves through $x$ for $\cL$ and $\cL'$ respectively.

\begin{lemma}\label{l.genus} $\cL'$ satisfies: 
\begin{enumerate}
\item\label{g1} $\cL'$ is minimal;
\item\label{g2} $\cE(\cL_x)$ is homeomorphic to $\cE(\cL'_x)$ for every $x\in\mathcal{M}$;
\item\label{g3} every end of every leaf of $\cL'$ is accumulated by genus;
\item\label{g4} the generic leaf of $\cL'$ has a Cantor set of ends.
\end{enumerate}
\end{lemma}

We leave the proof of the Lemma to the end of the section. Since every possible space of ends is realized by a leaf of $\cL$, Lemma \ref{l.genus} implies that every possible classiying triple of the form $(\infty,\cE,\cE)$ is realized by a leaf of $\cL'$. Finally, notice that $\cL'$ admits a hyperbolic lamination structure by \cite{Candel} (every leaf of $\cL'$ is of infinite topological type). 

We need a definition before proving Lemma \ref{l.genus}. 

Given a solenoid $\mathcal{N}$, we say that a Cantor set $J\subseteq\mathcal{N}$ is \emph{a transverse section} of $\mathcal{N}$ if for some $r>0$ we have $$\bigcup_{x\in J}B(x,r)$$ is an open set and $B(x,r)\cap J=\{x\}$ for every $x\in J$ and if all leaves of $\cN$ intersect $J$. Note that the pseudogroup of holonomy restricted to $J$ is minimal if and only if $\cL$ is minimal (see \cite{candel-conlon}). 

\begin{proof}[Proof of Lemma \ref{l.genus}] Consider a transverse section $J$ of $\cL$ contained inside $\cL\setminus U$ (this is possible by the structure of Cantor bundle of $\cL$). Since $\cL$ is minimal, the holonomy pseudogroup acts minimally on $J$. Notice that removing $U$ does not affect the holonomy pseudogroup restricted to $J$ and therefore the holonomy pseudogroup of $\cL'$ restricted to $J$ is also minimal which implies \ref{g1}. 

Take $x\in\mathcal{M}$, to show that $\cE(L_x)$ is homeomorphic to $\cE(L'_x)$ consider an exhaustion of $L_x$ by compact connected subsurfaces with boundary $$S_1\subseteq S_2\subseteq\ldots\subseteq S_n\subseteq\ldots$$
such that $\partial S_i\cap U=\emptyset$ for every $i\in\N$ and such that different boundary components of $S_i$ correspond to different connected components of $L_x\setminus S_i$ (this can be done using the core tree construction of \cite{BWal}). Since $\partial S_i\cap U=\emptyset$ for every $i\in\N$ this induces a natural exhaustion of $L_x'$ $$T_1\subseteq T_2\subseteq\ldots\subseteq T_n\subseteq\ldots$$ which induces an homeomorphism between the inverse limits of the system of connected components of $L_x\setminus S_i$ and that of $L_x'\setminus T_i$ proving \ref{g2}.

Notice that by the minimality of $\cL$, every connected component of $L_x\setminus S_i$ intersect $U$ and therefore every connected component of $L_x'\setminus T_i$ has non-trivial genus which implies Condition \ref{g3}. Finally, since the generic leaf of $\cL$ has a Cantor set of ends, there exists $R\subseteq J$ a generic subset such that $\cE(L_x)$ is a Cantor set for every $x\in R$. Therefore, Condition \ref{g2} implies that $\cE(L'_x)$ is a Cantor set for every $x\in R$ which implies Condition \ref{g4}. 
\end{proof}

\subsection{Proof of Lemma \ref{l.cantor}}

From now on $(K_0,K_1)$ will be a pair satisfying condition $(\ast)$. We shall give a criterion to prove that a sequence of partitions $(\xi_i)_{i\in\N}$ is adapted to a pair $(K_0,K_1)$. Consider $(\mu_i)_{i\in\N}$ a sequence of partitions by clopen sets of $K_1$ separating points (i.e. for every $x,y\in K_1$ there exists $n\in\N$ with $y\notin\mu_n(x)$). We call a sequence with this property a \emph{separating sequence}.

\paragraph{\textbf{A criterion for adaptability  }} Consider a separating sequence $(\mu_i)_{i\in\N}$. Assume we have a sequence of partitions by clopen sets $(\xi_i)_{i\in\N}$ and a sequence of integers $k_n\to+\infty$ satisfying that for every $n\in\N$ we have:
\begin{itemize}
\item $\xi_1,\ldots,\xi_{k_n}$ satisfies conditions (A~\ref{c1}, \ref{c2}, \ref{c3}, \ref{c4}) in the definition of adapted sequence of partitions.
\end{itemize}
and
\begin{enumerate}[(\text{A} 1)']
          \setcounter{enumi}{4}
\item\label{c5'} If $A\in\xi_i$ for some $i\leq k_n$ then (defining $\xi_i=\emptyset$ for $i<0$ or $i>k_n$)\begin{center}$A\cap K_0\neq\emptyset$ if and only if $A\in\xi_{i-1}\cup\xi_{i+1}$;\end{center} 

\item\label{c6'} $\mu_n\prec \xi_{k_n}$
\end{enumerate}

Then, the sequence $(\xi_i)_{i\in\N}$ is adapted for $(K_0,K_1)$. To see this, notice that conditions (A~\ref{c1}, \ref{c2}, \ref{c3}, \ref{c4}) are automatic. On the other hand, since $k_n\to+\infty$ we have that (A~\ref{c5}) follows from condition (A~\ref{c5'})'. Finally, to check condition (A~\ref{c6}) take $x,y\in K_1$ and $n\in\N$ such that $y\notin\mu_n(x)$. By condition (A~\ref{c6'})' we have that $\mu_n\prec \xi_{k_n}$ and therefore $\xi_{k_n}(x)\subseteq\mu_n(x)$ in particular $y\notin\xi_{k_n}(x)$. 

From now on, we fix  a separating sequence $(\mu_i)_{i\in\N}$. 
\paragraph{\textbf{Two useful lemmas  }}
The proof of the following lemma is left to the reader:

\begin{lemma}\label{l.choto} Given a compact, perfect and totally disconnected space $K$ and $n\in\N$, there exists $\nu_0,\ldots,\nu_n$ partitions by clopen sets satisfying:
\begin{itemize}
\item $\nu_0=\{K\}$
\item $\nu_i\prec\nu_{i+1}$
\item For every $i<n$ and $A\in\nu_i$, there exists infinite clopen sets $\{B_1,B_2,B_3\}\subseteq\nu_{i+1}$ such that $A=B_1\sqcup B_2\sqcup B_3$.
\end{itemize}
\end{lemma}
The following lemma is the key for proving Lemma \ref{l.cantor} and its proof is postponed until the final section of this appendix.

\begin{lemma}[Partition lemma]\label{l.key} Consider finite partitions by clopen sets $\xi$ and $\mu$ such that all finite elements of $\xi$ are singletons. Then there exists finite partitions by clopen sets $\xi_0,\xi_1,\dots,\xi_n$ satisfying:
\begin{enumerate}[(\text{B} 1)]
\item\label{f1} the sequence is increasing $\xi_0=\xi\prec\xi_1\prec\ldots\prec\xi_n$;

\item\label{f2} If $A\in\xi_i\setminus\xi_{i+1}$ then, there exists three different and non-empty elements $B_1,B_2,B_3\in\xi_{i+1}$ such that $A=B_1\sqcup B_2\sqcup B_3$;

\item\label{f3} If $A\in\xi_i$ for some $i\leq n$ then (defining $\xi_i=\emptyset$ for $i<0$ or $i>n$)\begin{center}$A\cap K_0\neq\emptyset$ if and only if $A\in\xi_{i-1}\cup\xi_{i+1}$ for $i=0,\ldots,k_n$;\end{center} 
\item\label{f4} $\mu\prec\xi_n$
\item\label{f5} finite elements of $\xi_n$ are singletons.

\end{enumerate}
\end{lemma}
\subsection{Partition lemma \ref{l.key} implies Lemma \ref{l.cantor}} We proceed to construct $(\xi_i)_{i\in\N}$, an adapted sequence of partitions for $(K_0,K_1)$ satisfying the criterion for adaptability stated above. 

Define $\xi^0_0=\{K_1\}$ and $\xi^0_1=\{L_1,L_2\}$ with $L_1$ and $L_2$ infinite clopen sets (the exact same argument works if we choose $\xi^0_1=\{L_1,L_2,L_3,L_4\}$ with $L_1,L_2,L_3,L_4$ infinite and clopen). Since $L_1$ and $L_2$ are infinite we can apply Lemma \ref{l.key} to the partitions $\xi_1^0$ and $\mu_1$. This gives a sequence $\xi^0_1=\xi^1_0\prec\xi^1_1\prec\ldots\prec\xi^1_{m_1}$ that in particular satisfies
\begin{itemize}
\item finite elements of $\xi^1_{m_1}$ are singletons;
\item $\mu_1\prec\xi^1_{m_1}$
\end{itemize}

Therefore, $\xi^1_{m_1}$ satisfies the hypothesis of Lemma \ref{l.key}.  Repeating this procedure we apply \ref{l.key} infinitely many times and get a family of finite sequences $((\xi^k_i)_{i\leq m_k})_{k\in\N}$ satisfying
\begin{itemize}
\item $\xi^k_0=\xi^{k-1}_{m_{k-1}}$ for $k\geq 1$
\item $\mu_k\prec\xi^k_{m_k}$ for $k\geq 1$
\item Items (B \ref{f2}, \ref{f3}) from the conclusion of Lemma \ref{l.key}.
\end{itemize}

Concatenating $\xi^0_0,\xi^0_1$ with the families $(\xi^k_i)_{1\leq i\leq m_k}$ we obtain a sequence $(\xi_n)_{n\in\N}$ for $(K_0,K_1)$ satisfying the criterion of adaptability (with the sequence $k_n=\sum_{k=1}^n m_k$). Indeed, items (A \ref{c1}, \ref{c2}, \ref{c3}) are clearly satisfied. Items (B \ref{f2}) and (B \ref{f3}) are guaranteed in the whole construction of $(\xi_n)_{n\in\N}$: so the sequence $(\xi_n)_{n\in\N}$ satisfies (A \ref{c4}) and (A \ref{c5'})'. Finally by construction of $k_n$ and we have $\mu_n\prec\xi_{k_n}$ for all $n$, providing (A \ref{c6'})'.

\subsection{Proof of the Partition lemma \ref{l.key}} 
Before starting the proof of Lemma \ref{l.key} we need to introduce some notations and definitions. 

\paragraph{\textbf{Pre-partitions }} We say that $\eta\subseteq \mathcal{P}(X)$ is a \emph{pre-partition} of $X$ if different elements of $\eta$ have empty intersection. Denote
$$\hat{\eta}:=\bigcup_{A\in\eta}A,$$
and note that if $\hat \eta=X$ then $\eta$ is a true partition of $X$. All the pre-partitions that we will consider will consist of finitely many clopen subsets. 

Given two pre-partitions $\eta_1,\eta_2$ of the same set $X$ we say that $\eta_2$ is \emph{finer than} $\eta_1$, and we write $\eta_1\prec\eta_2$, if
\begin{enumerate}
\item for every $A\in\eta_1$ we have that 
\begin{itemize}
\item either $A\subseteq \hat{\eta}_{2}$ or
\item $A\cap \hat{\eta}_{2}=\emptyset$
\end{itemize}
 \item $\eta_2(x)\subseteq\eta_1(x)$ for every $x\in\hat{\eta}_2$
\end{enumerate}

We say that a sequence of pre-partition $\eta=(\eta_0,\ldots,\eta_n)$ is \emph{increasing} if $\eta_0\prec\ldots\prec\eta_n$. We define the \emph{partition by minimal elements} associated to such an increasing family as $$P=\{A:A\in\eta_i\text{ and }A\cap\hat{\eta}_{i+1}=\emptyset\text{ for some }i\}.$$
Here we set $\hat \eta_{n+1}=\emptyset$ so $\hat P$ is the union of $\hat\eta_n$ with finitely elements called \emph{stopping elements}. It is clear that $P$ is a partition of $\hat \eta_0$, which induces the partition $\eta_n$ on the set $\hat \eta_n$.

Finally we define the depth function of the increasing family as the function \begin{center}$o:P\to\N$ satisying $A\in\eta_{o(A)}$ and $A\cap\hat{\eta}_{o(A)+1}=\emptyset$. \end{center}
Note $o(A)<n$ if and only if $A$ is a stopping element.


\paragraph{\textbf{Subdividing elements of $\xi$ }} Now we are ready to start the proof of Lemma \ref{l.key}. Consider $\xi$ and $\mu$ as in the hypothesis of that lemma. The first step is to subdivide every elements of $\xi$ in a way similar to Lemma \ref{l.key}. Given $A\in\xi$ define $$\mu_A:=\{A\cap B:B\in\mu\}$$ Since finite elements of $\xi$ are singletons, up to sub-dividing some elements we can assume that finite elements of $\mu_A$ are singletons and $\#\mu_A$ is odd. We now define a monotone family of pre-partitions $$\eta^A=(\{A\},\eta^A_1,\ldots,\eta^A_{n_A})$$ for every $A\in\xi$. For this enumerate $\mu_A=\{B_1,\ldots,B_{2m+1}\}$ and note $B_i^{\ast}=\cup_{j\geq i}B_j$. In order to construct our family we proceed inductively and discuss several cases.\\

\noindent \emph{Case 1. $\# \mu_A=1$.} In this case define $\eta^A:=(\{A\})$.\\

\noindent \emph{Case 2. $\# \mu_A=3$.} In this case, we have two possibilities

\begin{itemize} 
\item if $A\cap K_0=\emptyset$ define $\eta^A:=(\{A\},\{B_1,B_2,B_3\})$. 
\item if $A\cap K_0\neq\emptyset$ define $\eta^A:=(\{A\},\{A\},\{B_1,B_2,B_3\})$
\end{itemize}

\noindent \emph{Case 3. $\# \mu_A>3$.} In this case we construct $\eta^A$ concatenating some monotone families of prepartition. For this, given \emph{an odd integer} $i\leq 2m-1$ define:
\begin{itemize}
\item $\nu^i_1=\{B_i^{\ast}\}$, $\nu^i_2=\{B_i,B_{i+1},B_{i+2}^{\ast}\}$ and $\nu^i=(\nu^i_1,\nu^i_2)$ if $B_i^{\ast}\cap K_0\neq\emptyset$;
\item $\nu^i_1=\{B_i,B_{i+1},B_{i+2}^{\ast}\}$ and $\nu^i=(\nu^i_1)$ if $B_i^{\ast}\cap K_0=\emptyset$.
\end{itemize}
Finally, define $\eta^A$ as the concatenation of $\{A\},\nu^1,\nu^3,\ldots,\nu^{2m-1}$. 

By construction, it is clear that $\eta^A$ is an increasing family of finite pre-partitions of $A$ by clopen sets. We denote $\eta^A=(\eta^A_0,\ldots,\eta^A_{n_A})$ where
$$n_A=m=\frac{\# \mu_A-1}{2}.$$
It is also clear that $\mu_A$, which is finer than $\mu$, is the partition by minimal elements of $A$ associated to $\eta^A$.

\paragraph{\textbf{The joint sequence of pre-partitions }} We now put together the sequences of pre-partitions of elements of $\xi$ constructed above in a coherent way. Define
$$\eta_i:=\bigcup_{i\leq n_A}\eta^A_i.$$
Since for two differents elements $A$ and $A'$ of $\xi$ we have $\hat \eta^A_i\cap\hat \eta^{A'}_i=\emptyset$, $\eta_i$ is a pre-partition. Moreover, it follows directly from the construction of prepartitions $\eta^A_i$ that they form an increasing family of pre-partitions starting at $\xi$ that we call the \emph{joint sequence of pre-partitions}. Define $N:=\textrm{max}\{n_A:A\in \xi\}$ and set $\eta=(\eta_0,\ldots,\eta_N)$. 

By the remark above the partition by minimal elements $P$ associated to $\eta$, which is a partition of $\hat\xi=K_1$, is finer than $\mu$.

Now we transform our increasing family of pre-partitions $\eta$ into a family of actual partitions satisfying the thesis of Lemma \ref{l.key} by subdividing element of $P$ thanks to the process of ``natural continuation'' that we define below.

\paragraph{\textbf{The ``natural'' continuations }} Given a clopen set $C\subseteq K_1$ and $i<N+1$ we define the \emph{natural continuation} of $C$ \emph{associated to $i$ and $N+1$} as the family of pre-partitions $\nu^C=(\nu^C_0,\ldots,\nu^C_{N+1})$ obtained as follows.\\

\noindent \emph{Case 1. $C\cap K_0\neq\emptyset$.} In this case we define 
\begin{itemize}
\item $\nu^C_j=\emptyset$ for $j\leq i$
\item $\nu^C_j=\{C\}$ for $j>i$
\end{itemize}

\noindent \emph{Case 2. $C\cap K_0=\emptyset$.} In this case, since $(K_0,K_1)$ satisfies condition $(\ast)$ and $C$ is clopen, we deduce that $C$ does not contain isolated points and therefore it is perfect. Then, we can apply Lemma \ref{l.choto} to $C$ and $N+1-i$ to obtain an increasing family of (actual) partitions of $C$ $(\{C\}=\delta_0,\ldots,\delta_{N+1-i})$. Recall that every $\delta_j$ is obtained from $\delta_{j-1}$ by dividing it in three different infinite clopen sets. Finally define :
\begin{itemize}
\item $\nu^C_j=\emptyset$ for $j\leq i$
\item $\nu^C_{i+k}=\delta_k$ for $k=1,\ldots,N+1-i$
\end{itemize}

\paragraph{\textbf{The construction of the family $(\xi_0,\ldots,\xi_{N+1})$ }} Consider $P$ the partition by minimal elements of the increasing family of pre-partitions $\eta=(\eta_0,\ldots,\eta_N)$. As we saw above $P$ is a partition of $K_1$. Then, for every $C\in P$ consider the natural continuation $\nu^C$ of $C$ that is associated to the integers $o(C)$ and $N+1$. Now define $$\xi_i:=\eta_i\cup \bigcup_{C\in P} \nu^C_i$$ First notice that by the definition of the partition $P$, the increasing family $\eta$ and the families $\nu^C$, this union is indeed a disjoint union. It follows directly from the construction that  $(\xi_0,\ldots,\xi_{N+1})$ is an increasing family of genuine partitions of $K_1$. Furthermore we have by construction
$$\mu\prec P\prec\xi_{N+1}.$$

To check that finite elements of $\xi_{N+1}$ are singletons notice that (by the construction of the $\mu_A$) the finite elements appearing in the pre-partitions $\eta^A_i$ are indeed singletons. On the other hand, the natural continuations do not create new finite subsets of $K_1$. 

Condition (B \ref{f2}) in the thesis follows from the definition of the increasing famlies $\{\eta^A:A\in\xi\}$ and the natural continuations $\{\nu^C:C\in P\}$. Finally we need to check that given $A\in\xi_i$ we have \begin{center} $A\cap K_0\neq\emptyset$ if and only if $A\in \xi_{i-1}\cup\xi_{i+1}$\end{center}
(where $\xi_i=\emptyset$ if $i<0$ or $i>N+1$). 

We need to discuss three cases. \\

\noindent \emph{Case 1. $A\in \eta_i$ and $A\notin P$. } In this case (following the notation used in the construction of $\eta^A$) we have that $A=B_i^{\ast}$ for some $i$, $B_i^{\ast}=B_i\cup B_{i+1}\cup B_{i+2}^{\ast}$ and  $$\{B_i,B_{i+1},B_{i+2}^{\ast}\}\subseteq \eta_{i+1}\subseteq \xi_{i+1}$$

\noindent \emph{Case 2. $A\in\eta_i$ and $A\in P$. } If $A\cap K_0=\emptyset$ it follows from the construction of $\eta^A$ and the definition of natural continuation. On the other hand, notice that we performed natural continuation up to $N+1$ (which is strictly greater than $o(A)$ for every $A\in P$). Therefore, if $A\cap K_0\neq\emptyset$ we have that $A\in \nu^A_{o(A)+1}\dans \xi_{o(A)+1}$. \\

\noindent \emph{Case 3. $A\in \nu^A_i$. } This case also follows from the definition of the natural continuation $\nu^A$.

This finishes the proof of Lemma \ref{l.key} and thus that of Lemma  \ref{l.cantor}.\\

\paragraph{\textbf{Acknowledgements}} It is a pleasure to thank Matilde Mart\'inez and Rafael Potrie for very fruitful discussions during the elaboration of this paper. Also, we want to thank Gilbert Hector for kindly communicating to us Blanc's thesis \cite{Blanc_these}. Last but not least, we wish to thank the anonymous referee for her/his valuable comments.


\bibliographystyle{plain}

\end{document}